\documentclass[11pt,a4paper]{article}

\usepackage{fullpage}
\usepackage{authblk}

\usepackage{colortbl}   
\usepackage{xcolor}     

\usepackage{tabularx}
\usepackage{threeparttable}
\usepackage{amsmath}
\usepackage{amssymb}
\usepackage{amsthm}
\usepackage{mathtools}
\usepackage{booktabs}
\usepackage{tikz}                      
\usetikzlibrary{positioning, arrows.meta, bending, calc} 
\usepackage{accents} 
\usepackage{orcidlink}
\usepackage{algorithm}
\usepackage{algpseudocode}
\usepackage{hyperref}

\definecolor{banana}{HTML}{ffbe0b}
\definecolor{mango}{HTML}{fb5607}
\definecolor{vividorchid}{HTML}{c11cad}
\definecolor{malinconicblu}{HTML}{8338ec}
\definecolor{skyblue}{HTML}{3a86ff}

\newcommand{\treecol}{vividorchid}

\newcommand{\ared}{vividorchid}
\newcommand{\bluecycle}{skyblue}

\newtheorem{theorem}{Theorem}[section]
\newtheorem{proposition}[theorem]{Proposition}
\newtheorem{corollary}[theorem]{Corollary}
\newtheorem{lemma}[theorem]{Lemma}
\theoremstyle{definition}
\newtheorem{definition}[theorem]{Definition}
\newtheorem{example}[theorem]{Example}
\theoremstyle{remark}

\newtheorem{observation}[theorem]{Observation}
\newcommand{\nc}{\newcommand}

\nc{\tm}[1]{\text{\(#1\)}}

\nc{\bR}{\mathbb{R}}
\nc{\cK}{\mathcal{K}}
\nc{\cB}{\mathcal{B}}
\nc{\bZ}{\mathbb{Z}}

\nc{\nullv}{\mathbf{0}}
\nc{\card}[1]{\tm{\left|#1\right|}}

\nc{\cyclematrix}{\Gamma}
\nc{\forbiddenminors}{\mathcal{X}}
\nc{\odigraph}{\overline{\digraph}}
\nc{\oweight}{\overline\weight}
\nc{\otree}{\overline\tree}
\newcommand{\quotient}[2]{{\raisebox{.2em}{$#1$}\left/\raisebox{-.2em}{$#2$}\right.}}
\nc{\ocyclebasis}{\overline\cyclebasis}

\nc{\graph}{\tm{G}}
\nc{\digraph}{\tm{D}}
\nc{\vertices}[1]{\tm{V(#1)}}
\nc{\Arcs}[1]{A(#1)}
\nc{\Dvertices}{\vertices{\digraph}}
\nc{\Darcs}{\Arcs{\digraph}}
\nc{\Arc}{e}
\nc{\Arcp}{\Arc'}
\nc{\Arco}{\overline{\Arc}}
\nc{\outn}[1]{\text{\(\delta^+(#1)\)}}
\nc{\inn}[1]{\text{\(\delta^-(#1)\)}}
\nc{\eardecomp}{\mathcal{E}}
\nc{\walk}{P}
\nc{\walkk}{Q}
\nc{\walkh}{\walkk}
\nc{\ear}{P}
\nc{\subgraph}{\digraph}
\nc{\subbgraph}{\overline{\digraph}}
\nc{\block}{B}
\nc{\blocks}{\mathbf{B}}
\nc{\strongcomponent}{\mathcal{D}}
\newcommand{\vertex}{v}
\nc{\cT}{\mathcal{T}}
\nc{\tree}{\cT}
\nc{\treec}{\cT^*}
\nc{\startree}{S}

\nc{\cyclespace}[2]{\mathcal{C}_{#2}(#1)}
\nc{\Dcyclespace}{\cyclespace{\digraph}{\bR}}
\nc{\field}{\cK}
\nc{\cycle}{C}
\nc{\cyclebasis}{\tm{\cB}}
\nc{\cycleo}{\overline{\cycle}}
\nc{\circuit}{\widehat{\cycle}}
\nc{\forwardcycles}{\mathcal{F}}
\newcommand{\simplecycle}[1]{S^{(#1)}}
\newcommand{\weight}{w}
\nc{\w}{w}
\nc{\tweight}{\widetilde{\weight}}
\nc{\cweight}{\lambda}

\newcommand{\treecycleT}[1]{\cycle_{\tree}^{(#1)}}
\nc{\eincycles}[1]{\mathcal{H}(#1)}
\nc{\stcyclebasis}{\cyclebasis^*}
\nc{\stcycle}{\cycle^*}

\newcommand{\bigo}[1]{\mathcal{O}\left(#1\right)}
\nc{\ArcC}{e^{(\cycle)}}
\nc{\RarcC}[2]{#1^{(#2)}}
\nc{\simplecycles}{\mathcal{S}}
 
\newcommand*\Rop[1]{\mkern2mu\accentset{\leftrightarrow}{#1}\mkern2mu}%
\nc{\Rdigraph}{\Rop{\digraph}}
\nc{\Rvertices}{V}
\nc{\Rarcs}{A'}
\nc{\Rcycle}{\Rop{\cycle}}
\nc{\Rarc}{\Rop{\Arc}}
\nc{\Rcyclebasis}{\Rop{\cyclebasis}}

\nc{\head}{\vertex}
\nc{\tail}{u}
\newcommand{\cyclerank}[1]{\tm{\mu(#1)}}
\newcommand{\Dcyclerank}{\tm{\mu}}
\newcommand{\Rcyclerank}{\tm{\Rop{\mu}}}
\newcommand{\optcost}[2]{\mathrm{OPT}^{#1}(#2)}
\newcommand{\MWFCB}[1]{\optcost{\mathrm{wF}}{#1}}
\newcommand{\MWFFCB}[1]{\optcost{\mathrm{wF,fwd}}{#1}}
\nc{\projection}{\pi}

\newcommand{\minimalcyclesk}{\mathfrak{C}_k}
\newcommand{\proj}{\mathrm{proj}}

\newcommand{\completegraph}[1]{K_{#1}}
\nc{\kfive}{\completegraph{5}}
\nc{\mcbeq}{opt-in}
\nc{\notmcbeq}{opt-out}
\nc{\optbases}{\mathfrak B_O}
\nc{\optintbases}{\mathfrak B_I}
\nc{\optnotintbases}{\mathfrak B_N}
\nc{\intcone}{\mathcal Q_I}

\nc{\aut}{\mathrm{Aut}}
\nc{\opennodes}{\mathcal O}

\nc{\cervec}{b}
\nc{\Cervec}{\mathcal{V}}
\nc{\domvec}{\cervec}

\nc{\tri}{T}

\DeclareMathOperator*{\argmin}{arg\,min}
\DeclareMathOperator*{\supp}{supp}
\DeclareMathOperator*{\Span}{Span}

\begin{document}

\title{Characterizations and Complexity of Minimum Forward and Integer Cycle Bases}

\author[1]{Gabor Riccardi \orcidlink{0009-0009-1616-6359}}
\author[2]{Niels Lindner \orcidlink{0000-0002-8337-4387]}}
\affil[1]{Dipartimento di Matematica ``F. Casorati'', University of Pavia, Via Adolfo Ferrata, 5, 27100, Pavia, Italy}
\affil[2]{Department of Mathematics and Computer Science, Freie Universität Berlin, c/o~Zuse~Institute Berlin, Takustr.\ 7, 14195, Berlin, Germany}



\maketitle

\begin{abstract}
The cycle space of a directed graph is generated by a cycle basis, where, in general, cycles are allowed to have both forward and backward arcs. In a forward cycle, all arcs must follow the given direction. Several open questions remain regarding the complexity of the minimum cycle basis problem for various families of cycle bases, in particular the minimum-weight integral cycle basis problem, and the minimum-weight weakly and strictly fundamental forward cycle basis problems. In this paper, we address these open questions.

First, we study the existence, structure, and computational complexity of minimum-weight forward cycle bases in directed graphs. We give a complete structural characterization of digraphs that admit weakly fundamental (and hence integral) forward cycle bases, showing that this holds if and only if every block is either strongly connected or a single arc. We further provide an easily verifiable characterization of when a strongly connected digraph admits a forward fundamental cycle basis, proving that such a basis exists if and only if the set of directed cycles has cardinality equal to the cycle rank; in this case, the basis is unique and computable in polynomial time, and nonexistence can likewise be certified efficiently. Lastly, we show that while minimum-weight forward fundamental cycle bases can be found in polynomial time whenever they exist, the minimum-weight forward weakly fundamental cycle basis problem is APX-hard via an L-reduction from the minimum-weight weakly fundamental cycle basis problem on digraphs with metric weights.

Second, we introduce the notion of \mcbeq\ graphs: the family of graphs for which minimum cycle bases are integral for any weight function. We show that the family of \mcbeq\ graphs is minor-closed and hence, by the Robertson--Seymour theorem, is characterized by a finite set of forbidden minors, yielding a non-constructive proof that the \mcbeq\ recognition problem is solvable in polynomial time. Lastly, we present an algorithm to check whether a graph is \mcbeq, and if not, to identify which of its minors belong to the set of forbidden minors. Applying this algorithm, we show that the complete graph \(\completegraph n\) is \mcbeq\ if and only if \(n\le7\).
\end{abstract}

\vspace{\baselineskip}
\noindent
\emph{Keywords:} cycle bases, forward cycle bases, fundamental cycle bases, weakly fundamental cycle bases, integer cycle bases, minimum cycle basis, complexity

\section{Introduction}

Cycle bases provide a compact representation of all cycles in a graph and play a key role in optimization problems such as periodic scheduling, electrical network analysis, and modeling of chemical and biological pathways (see, e.g., Kavitha et al.\ \cite{cyclebasisintro}). Different optimization contexts call for different classes of cycle bases. For instance, aiming at railway timetable optimization, Nachtigall \cite{Nachtigall1996CuttingPlanes} introduced a mixed-integer programming formulation of the \emph{periodic event scheduling problem} (PESP) based on fundamental cycle bases, which was later extended to integral cycle bases by Liebchen and Peters \cite{LIEBCHEN200998} and to forward cycle bases by Lindner et al.\ \cite{lindner} and Masing et al.\ \cite{masing_forward_2023}. The choice of cycle basis can significantly influence computational efficiency, motivating the study of minimum-weight cycle basis problems for each class of cycle bases.

While existence and complexity results are well established for many classes of cycle bases \cite{BangJensenGutin2009}, an exception is the computational complexity of the minimum-weight integral cycle basis problem, which remains open. Forward cycle bases are less well understood. In \cite{forwardGleiss}, Gleiss, Leydold, and Stadler show that a digraph admits a forward cycle basis if and only if each block is either strongly connected or a single arc, and that a minimum-weight forward cycle basis can be computed in polynomial time. Masing et al.\ \cite{masing_forward_2023} show that forward cycle bases can improve the efficiency of branch-and-cut in integral cycle basis formulations of PESP. They also prove that any graph admits an orientation with a strictly fundamental forward cycle basis, and that a forward integral cycle basis always exists for so-called line-based event-activity networks constructed from a public transport line plan.

Several open questions remain regarding the structure and complexity of both integral and forward cycle bases. These include identifying structural graph properties that guarantee the existence of integral, weakly fundamental, and strictly fundamental forward cycle bases, establishing the computational complexity of deciding their existence, and addressing the optimization challenge of computing a minimum-weight basis within a given class. For instance, while the minimum integral cycle basis problem can be solved on planar graphs using Horton's algorithm over isometric cycles \cite{cyclebasisintro}, the broader graph classes for which this algorithmic approach remains valid are unknown. Delineating these structural boundaries represents a step toward settling the exact computational complexity of the problem. 
Finally, from an applied perspective, when a strict forward cycle basis cannot be guaranteed, a new optimization problem arises: determining a cycle basis that maximizes the total number of forward cycles.

The paper is organized as follows. In
Section~\ref{sec:preliminaries}, we introduce basic notions and
definitions related to directed graphs, cycle spaces, and various classes
of cycle bases. Section~\ref{sec:existence} presents results on the
existence of forward cycle bases and characterizations of weakly
fundamental and fundamental forward cycle bases. In
Section~\ref{sec:APX-hardness}, we establish the APX-hardness of the minimum-weight forward weakly fundamental cycle basis problem via an L-reduction from the (non-forward) minimum-weight weakly fundamental cycle basis problem with metric weights. In
Section~\ref{sec:integrality of minimum cycle bases}, we introduce
\mcbeq\ graphs -- those for which classical minimum cycle basis
algorithms return an integral cycle basis for any weight function -- and
show that this family is minor-closed and hence fully characterized by a
finite set of forbidden minors due to the Robertson--Seymour theorem
\cite{ROBERTSON1995}. Section~\ref{sec: results for X} establishes
several key structural properties of the set of forbidden minors \(\forbiddenminors\), and
Section~\ref{sec: forbidden minors alg} presents a recursive algorithm to
determine membership in \(\forbiddenminors\), which we use to fully
characterize \mcbeq-ness for complete graphs.
Finally, Section~\ref{sec:conclusion} summarizes our contributions and outlines open questions for future research.

This paper extends a previous conference paper \cite{DBLP:conf/inoc/RiccardiL26}. The APX-hardness result in Section~\ref{sec:APX-hardness} and all content in Sections~\ref{sec:integrality of minimum cycle bases}-\ref{sec: forbidden minors alg} are new.

\section{Preliminaries and basic notions}\label{sec:preliminaries}
Let $\digraph$ be a finite directed graph (digraph), where $\Dvertices$ is the vertex set and $\Darcs\subseteq \Dvertices\times\Dvertices$ is the arc set. 
An \emph{(undirected) walk} in $\digraph$ is a sequence $\walk = \vertex_0 \Arc_1 \vertex_1 \Arc_2 \dots \Arc_k \vertex_k$ such that for each $i=1,\dots,k$ the arc $\Arc_i$ is either $(\vertex_{i-1},\vertex_i)\in\Darcs$ or $(\vertex_i,\vertex_{i-1})\in\Darcs$ (i.e. the arc may be traversed in either direction). 
If every $\Arc_i=(\vertex_{i-1},\vertex_i)\in\Darcs$ we say $\walk$ is a \emph{directed walk}. 
The vertices $\vertex_0$ and $\vertex_k$ are the \emph{start} and \emph{end} of $\walk$, denoted $\mathrm{start}(\walk)$ and $\mathrm{end}(\walk)$. 
If $\mathrm{end}(\walk)=\mathrm{start}(\walkk)$, we denote by $\walk\walkk$ the \emph{concatenation} of walks $\walk$ and $\walkk$. 
A \emph{(directed) path} is a (directed) walk with all vertices distinct. 
A \emph{(directed) circuit} is a closed (directed) walk (i.e. $\mathrm{start}(\walk)=\mathrm{end}(\walk)$), and a \emph{(directed) cycle} is a circuit in which all vertices except the first and last are distinct. Whenever convenient, we identify a path with the ordered sequence of its arcs or vertices, rather than with its formal representation.
A digraph is \emph{strongly connected} if every ordered pair of vertices is joined by a directed path, and \emph{weakly connected} if there is a (possibly non-directed) path between every pair of vertices.

For a forest $\tree\subset\digraph$, i.e., a subgraph that does not contain any cycle, we write
\[
\treec \coloneqq \Darcs\setminus \Arcs{\tree}
\]
for the set of chords (arcs not in~$\tree$). If $\vertices{\tree} = \vertices{\digraph}$, then $\tree$ is \emph{spanning}. A weakly connected forest is a \emph{tree}.

If $\field$ is a field, $\cycle\in\field^{\Darcs}$ and $\Arc\in\Darcs$, then $\cycle(\Arc)$ denotes the coordinate of $\cycle$ corresponding to the arc $\Arc$. 
The \emph{support} of a vector $\cycle\in\field^{\Darcs}$ is denoted
\[
\supp(\cycle):=\{\Arc\in\Darcs:\ \cycle(\Arc)\neq 0\}.
\]
We often identify the vector \(\cycle\) with the set $\supp(\cycle)$ and write ``$\Arc\in\cycle$\,'' to mean $\Arc\in\supp(\cycle)$ (equivalently $\cycle(\Arc)\neq 0$). In particular, we interpret a cycle $\cycle$ in $\digraph$ as a vector in $\field^{\Darcs}$ that we also call $\cycle$ as follows: If $\cycle = \vertex_0 \Arc_1 \vertex_1 \Arc_2 \dots \Arc_k \vertex_k$ with $\vertex_0 = \vertex_k$, the \emph{incidence vector} is given by 
\[ 
\cycle(e) \coloneqq \begin{cases}
    1 & \text{ if } e = e_i \text{ for some } i \in \{1, \dots, k\} \text{ and } e_i = (v_{i-1}, v_i),\\
    -1 & \text{ if } e = e_i \text{ for some } i \in \{1, \dots, k\} \text{ and } e_i = (v_i, v_{i-1}),\\
    0 & \text{ otherwise.}
\end{cases}
\]

\begin{definition}[cycle space]
The \emph{cycle space} of \(\digraph\) over a field \(\field\) is the subspace of \(\field^{\Darcs}\) defined by:
\begin{equation}
    \cyclespace{\digraph}{\field} \coloneqq \left\{\cycle \in \field^{\Darcs} \,\middle|\, \sum _{\Arc \in \outn{v} } \cycle(\Arc) = \sum_{\Arc \in \inn{v}} \cycle(\Arc) \right\}.
\end{equation}
\end{definition}
The dimension of the cycle space of \(\digraph\) is called \emph{cycle rank} (cyclomatic number, cyclotomic number) and is denoted by \cyclerank{\digraph}. By means of its incidence vector, any cycle in $\digraph$ can be seen as an element of the cycle space.

\begin{definition}[classes of cycle bases]\label{def:cycle-basis-classes}
    A \emph{cycle basis} \(\cyclebasis\) of \(\digraph\) is a basis of $\cyclespace{\digraph}{\field}$ composed of incidence vectors of cycles in \(\digraph\). Moreover, we call \(\cyclebasis\):
    \begin{enumerate}
        \item \emph{directed}: if \(\cyclebasis\) is a cycle basis of \(\cyclespace{\digraph}{\bR}\).
        \item \emph{undirected}: 
        if \(\cyclebasis\) is a basis of \(\cyclespace{\digraph}{\mathbb F_2}\), where \(\mathbb F_2\) is the finite field with two elements.
        \item \emph{integral}: if any cycle \(\cycle\) of \(\digraph\) can be written as an integer linear combination of vectors in \(\cyclebasis\).
        \item \emph{weakly fundamental} (Whitney \cite{whitney_weakly_1935}) : if there exists an ordering of vectors in \(\cyclebasis = \{\cycle_1,\ldots, \cycle_{\mu}\}\) such that for all \(1\leq i \leq \mu\) there exists \(\Arc\in \cycle_i\) such that \(\Arc\notin \bigcup_{j <i}\cycle_j\).
        \item \emph{(strictly) fundamental:} if there exists a spanning forest \(\tree\) in \(\digraph\) such that \(\cyclebasis = \{\treecycleT{\Arc}\mid \Arc\in \treec\)\}, where \(\treecycleT{\Arc}\) is the only cycle contained in \(\tree\cup\{\Arc\}\).
    \end{enumerate}
\end{definition}

For an in-depth introduction to cycle spaces and bases, we refer to \cite{cyclebasisintro}, from which we adopt the same notation. 
Definition~\ref{def:cycle-basis-classes} in fact provides a hierarchy: Any fundamental cycle basis is weakly fundamental, any weakly fundamental cycle basis is integral, any integral cycle basis is undirected, and any undirected cycle basis is directed \cite{cyclebasisintro}. 




For $\cycle \in \Dcyclespace$ and $e \in \supp(\cycle)$, we can either have $\cycle(e) > 0$ (``$e$ is forward'') or $\cycle(e) < 0$ (``$e$ is backward''). We write $\cycle \geq \nullv$ if $\cycle(e) > 0$ for all $e \in \supp(C)$, i.e., if $\cycle$ contains only forward arcs. In particular, this holds for any directed cycle. We further define
    \[ \forwardcycles(\digraph) \coloneqq \{ \cycle \mid \cycle \text{ is a directed cycle in } \digraph \}. \]

\begin{definition}[forward cycle basis]
    A cycle basis $\cyclebasis$ of \(\cyclespace{\bR}{\digraph}\) is a \emph{forward cycle basis} if all cycles \(\cycle\) in \(\cyclebasis\) are directed.
\end{definition}

We prefer the term \emph{forward}, since ``directed cycle basis'' is used in \cite{cyclebasisintro} and in Definition~\ref{def:cycle-basis-classes} for an arbitrary cycle basis of $\cyclespace{\bR}{\digraph}$, although the latter need not be composed of directed cycles.
In the same way as arbitrary cycle bases, forward cycle bases can be categorized as directed, undirected, integral, weakly fundamental, or fundamental.

Lastly, the notion of separability is essential to characterize digraphs that admit a forward cycle basis:

\begin{definition}[separable digraph]
    A digraph \(\digraph\) is \emph{separable} if there exists a vertex \(\vertex \in \Dvertices\) such that the digraph \(D \setminus \{\vertex\}\) is not weakly connected.
\end{definition}

If a graph is separable, its maximal non-separable subgraphs form a partition and are called \emph{blocks} of \(\digraph\). We briefly note that the cycle space of a digraph naturally decomposes into the direct sum of the cycle spaces of its blocks:

\begin{observation}\label{obs:blocks}
    Let $\blocks$ be the set of blocks of a digraph $\digraph$. Then, for any field $\field$,
    \[ \cyclespace{\digraph}{\field} = \bigoplus_{\block \in \blocks} \cyclespace{\block}{\field}. \]
\end{observation}

In particular, a directed/undirected/integral/weakly fundamental/fundamental cycle basis for $\digraph$ can be constructed by assembling directed/undirected/integral/weakly fundamental/fundamental cycle bases for each block.

To answer the question regarding the existence of forward weak\-ly fundamental and integral cycle bases, we will use the concept of ear decomposition of a directed graph. A (directed) ear \(\ear \subset \digraph\) of a subgraph \(\subbgraph \) of \( \digraph \) is a (directed) path or cycle in \(\digraph\) such that only the first and last vertex of \(\ear\) are in \(\subbgraph\) while all the interior vertices of \(\ear\) are not in \(\subbgraph\).

\begin{definition}[(directed) ear decomposition]
    A \emph{(directed) ear decomposition} of a digraph \(\digraph \) with at least two vertices is a sequence of nested subgraphs \(\subgraph_1\subset\ldots\subset\subgraph_t = \digraph\) such that
    \begin{enumerate}
        \item \(\subgraph_1\) is a (directed) cycle,
        \item \(\subgraph_i = \subgraph_{i-1} \cup \ear_i\), where \(\ear_i\) is a (directed) ear of \(\subgraph_{i-1}\) for \(1 < i \leq t\).
    \end{enumerate}
\end{definition}

It is a known result in graph theory that a digraph admits a directed ear decomposition if and only if it is strongly connected.

\begin{proposition}[\cite{BangJensenGutin2009}]\label{prop:directed-ear}
A strongly connected digraph \(\digraph\) admits a directed ear decomposition \(\subgraph_1\subset\ldots\subset\subgraph_t = \digraph\). Furthermore,  \(t = \mu(D) = |\Arcs{D}| -  |\Dvertices|+1\) and the ear decomposition can be computed in \(\bigo{\min(\card{\Dvertices},\mu(\digraph)) \cdot (\card{\Dvertices}+\card{\Darcs})}\)
 time. 
\end{proposition}

\section{Results on the existence of forward cycle bases}\label{sec:existence}
In this section, we answer questions regarding the existence of integral, weakly fundamental, and strictly fundamental forward cycle bases.
\begin{proposition}\label{prop: weakly fundamental forward cycle basis}
{Let \(\digraph\) be a digraph. The following are equivalent:
\begin{enumerate}\label{prop: existence of weakly fundamenta forward cycle basis}
    \item\label{cond: exist wff} \(\digraph\) admits a forward weakly fundamental cycle basis.
    \item\label{cond: exist if} \(\digraph\) admits a forward integral cycle basis.
    \item\label{cond: exist u} \(\digraph\) admits a forward undirected cycle basis.
    \item\label{cond: exist f} \(\digraph\) admits a forward directed cycle basis.
    \item\label{cond: block}  Each block of \(\digraph\) is strongly connected or a single arc.
\end{enumerate}
}
\end{proposition}
\begin{proof}
The equivalence of \eqref{cond: exist f} and \eqref{cond: block} is Theorem~7 in \cite{forwardGleiss}. The implications \eqref{cond: exist wff} $\Rightarrow$\eqref{cond: exist if} $\Rightarrow$ \eqref{cond: exist u} $\Rightarrow$ \eqref{cond: exist f} are clear.
Now assume that \eqref{cond: block} holds. 
By Observation~\ref{obs:blocks}, if there exists a weakly fundamental cycle basis for each block of $\digraph$, a forward weakly fundamental cycle basis can be naturally obtained for \(\digraph\). Therefore, we only need to show that a weakly fundamental cycle basis exists for a block $\block$ of $\digraph$. The case of a single arc is trivial. Otherwise, \(\block\) is strongly connected. Therefore it admits a directed ear decomposition \(\subgraph_1 \subset \ldots\subset \subgraph_t\) by Proposition~\ref{prop:directed-ear}. 

From this ear decomposition, we can construct a weakly fundamental cycle basis as follows: Let \(\cycle_1 \coloneqq \subgraph_1\). 
For \(i > 1\), consider the ear \(\ear_i\) and let \(\vertex,\vertex' \in \vertices{\ear_i}\) be the first and last vertex of \(\ear_i\).
Since \(\subgraph_{i-1}\) is strongly connected, there exists a directed walk \(\walk'_i\) from \(\vertex'\) to \(\vertex\) in \(\subgraph_{i-1}\). 
Since \(\vertices{\walk'_i} \subset \subgraph_{i-1}\) and \(\ear_i\) is an ear of \(\subgraph_{i-1}\), the only vertices in \(\vertices{\ear_i}\cap\vertices{\walk'_i}\) are \(\vertex\) and \(\vertex'\), thus \(\cycle_i \coloneqq \ear_i\walk'_i \subset \subgraph_i\) is a directed cycle.
The set 
 \(\cyclebasis \coloneqq \{\cycle_i\}_{i=1}^t\) is a weakly fundamental cycle basis of \(\block\), since \(t = \card{\Arcs{\block}}-\card{\vertices{\block}}+1\), and for all \(i > 1\), there is an edge \(e \in \Arcs{\ear_i} \subset \Arcs{\subgraph_{i}} \setminus \Arcs{\subgraph_{i-1}} \) that is not contained in \(\cup_{j=1}^{i-1}\cycle_j \subset \subgraph_{i-1}\). 
\end{proof}

In particular, every strongly connected digraph $\digraph$ has a forward weakly fundamental cycle basis, as it admits a directed ear decomposition (Proposition~\ref{prop:directed-ear}).

Similarly, when not all blocks of \(\digraph\) are either a single arc or strongly connected, we can consider the following extremal cycle basis problem \cite{masing_forward_2023}: Find a cycle basis such that the number of directed cycles is maximum. We obtain:

\begin{proposition}\label{prop:max-forward-cycles}
Let \(\digraph\) be a digraph, let \(\cyclebasis\) be a directed, undirected, integral, or weakly fundamental cycle basis  of \(\digraph\) with a maximum number of directed cycles, and let \(\strongcomponent\) be the set of strongly connected components of \(\digraph\). Then 
\begin{equation}\label{eq:max-forward-cycles}
\card{\cyclebasis \cap \forwardcycles(\digraph)} = \sum_{\subbgraph \in \strongcomponent}\mu(\subbgraph).
\end{equation}
\end{proposition}
\begin{proof}
    Any directed cycle in $\digraph$ is contained in some strongly connected component of $\digraph$. In each component $\subbgraph \in \strongcomponent$, we can have at most $\mu(\subbgraph)$ linearly independent cycles, so that for any cycle basis $\cyclebasis$ of $\digraph$ holds $\card{\cyclebasis \cap \forwardcycles(\digraph)} \leq \sum_{\subbgraph \in \strongcomponent}\mu(\subbgraph)$.

    We will now construct a weakly fundamental cycle basis $\cyclebasis$ of $\digraph$ such that \eqref{eq:max-forward-cycles} holds. At first, we choose a spanning tree of each strongly connected component of $\digraph$. We then connect these trees to a spanning forest of $\digraph$ and consider the corresponding fundamental cycle basis $\cyclebasis'$ of $\digraph$. Any strongly connected digraph admits a forward cycle basis and hence a forward weakly fundamental cycle basis by Proposition~\ref{prop: weakly fundamental forward cycle basis}. In $\cyclebasis'$, for each strongly connected component $\subbgraph \in \strongcomponent$, we consider the cycles that arise from chords in $\subbgraph$. These constitute a fundamental cycle basis of $\subbgraph$, which we replace by a forward weakly fundamental cycle basis of $\subbgraph$. This way, we can transform $\cyclebasis'$ to a weakly fundamental cycle basis $\cyclebasis$ of $\digraph$ that contains precisely $\sum_{\subbgraph \in \strongcomponent}\mu(\subbgraph)$ directed cycles.
\end{proof}

A different proof for Proposition~\ref{prop: weakly fundamental forward cycle basis} and \ref{prop:max-forward-cycles}, using a notion of cycle subspaces and a contraction argument, is given by Masing \cite{masing_diss}.

As noted in \cite{masing_forward_2023}, not all digraphs admit a forward fundamental cycle basis. We now show an easily verifiable characterization of digraphs which admit a forward fundamental cycle basis.

We will make use of the following lemma that states that every element of the cycle space with forward arcs only lives in the cone generated by the directed cycles.

\begin{lemma}[\cite{BangJensenGutin2009}]
\label{lemma:generators of forward cycles}
Let $\digraph$ be a directed graph and let $\cycleo \in \Dcyclespace$ with $\cycleo \geq \nullv$. Then 
there exist nonnegative scalars $\{\lambda_{\cycle}\}_{\cycle\in\forwardcycles(\digraph)}$ such that
\[
\cycleo \;=\; \sum_{\cycle\in\forwardcycles(\digraph)} \lambda_{\cycle}\,\cycle.
\]
Moreover, if $\cycleo$ has integer coordinates, then the coefficients $\lambda_{\cycle}$ can be chosen to be integers.
\end{lemma}

We can now give our surprisingly simple characterization.

\begin{theorem}\label{thm: forward fundamental existence}
    Let \(\digraph\) be a digraph such that each block is strongly connected or a single arc. Then $\digraph$ admits a forward fundamental cycle basis if and only if \[ |\forwardcycles(\digraph)| = \cyclerank{\digraph}. \]
    In this case, $\forwardcycles(\digraph)$ constitutes the unique forward cycle basis of $\digraph$.
\end{theorem}
\begin{proof}
 Assume that a forward fundamental cycle basis $\cyclebasis$ exists.  Let $\tree\subset\digraph$ be the spanning forest that induces the basis, and let
\(\treec\) denote the set of chords of \(\tree\).  Then the cycle basis is
\[
\cyclebasis=\{\treecycleT{\Arc} \mid \Arc\in\treec\},
\]
where $\treecycleT{\Arc}$ denotes the unique cycle contained in $\tree\cup\{\Arc\}$.  
Note that for all \(\cycle\in\forwardcycles(\digraph)\) and all \(\Arc \in \Darcs\), we have \(\cycle(\Arc)=1\) if and only if \(\Arc\in\cycle\) and \(\cycle(\Arc)=0\) otherwise. Since $\cyclebasis\subset\forwardcycles(\digraph)$ and
\(
\card{\cyclebasis}=\cyclerank{\digraph},
\)
it suffices to show $\forwardcycles(\digraph)\subset\cyclebasis$.

Let $\cycleo\in\forwardcycles(\digraph)$.  As $\{\treecycleT{\Arc}\}_{\Arc\in\treec}$ is an integral cycle basis, there exist integers $\{\lambda_{\Arc}\}_{\Arc\in\treec}$ with
\[
\cycleo=\sum_{\Arc\in\treec}\lambda_{\Arc}\,\treecycleT{\Arc}.
\]
For each chord $\Arco\in\treec$ the arc $\Arco$ belongs only to the cycle $\treecycleT{\Arco}$ among the cycles in \(\cyclebasis\), and moreover $\treecycleT{\Arco}(\Arco)=1$ because $\treecycleT{\Arco}$ is directed.  Comparing the value of both sides on the arc $\Arco$ we obtain
\[
\cycleo(\Arco)=\sum_{\Arc\in\treec}\lambda_{\Arc}\,\treecycleT{\Arc}(\Arco)=\lambda_{\Arco}.
\]
Thus $\lambda_{\Arco}=1$ if $\Arco\in\cycleo$ and $\lambda_{\Arco}=0$ otherwise.  Hence $\cycleo$ is the sum of precisely those $\treecycleT{\Arco}$ that correspond to the chords contained in $\cycleo$.  If more than one such $\treecycleT{\Arco}$ appeared in the sum, the support of the sum would be a union of two or more arc-disjoint cycles. This contradicts that $\cycleo$ is a cycle.  Therefore exactly one coefficient is $1$ and the others are $0$, so $\cycleo=\treecycleT{\Arco}$ for some $\Arco\in\treec$.  Consequently $\cycleo\in\cyclebasis$, proving $\forwardcycles(\digraph)\subset\cyclebasis$ and hence \( \cyclebasis=\forwardcycles(\digraph) \).

Conversely, suppose
\(
\card{\forwardcycles(\digraph)}=\cyclerank{\digraph},
\)
i.e., the family of directed cycles has a size equal to the cycle rank. Since by Proposition~\ref{prop: existence of weakly fundamenta forward cycle basis}, a forward cycle basis $\cyclebasis$ exists and is contained in \(\forwardcycles(\digraph)\), it must be equal to \(\forwardcycles(\digraph)\). Given \(\cycleo\in\cyclebasis\) we define, for each \(\Arc\in\cycleo\), the set 
\[\eincycles{e} \coloneqq \{\cycle \in \cyclebasis\setminus\{\cycleo\} \mid e \in \cycle\}.\]
Using the standard characterization of fundamental cycle bases \cite{syslo_characterization_2006}, a cycle basis \(\cyclebasis\) is fundamental if and only if for every \(\cycleo \in \cyclebasis\) there exists an arc \(\Arco \in \cycleo\) such that
\[
\Arco \notin \bigcup_{\cycle \in \cyclebasis \setminus \{\cycleo\}} \cycle ,
\]
or equivalently, such that \(\card{\eincycles{\Arco}} = 0\). Let \(\Arco \in \argmin_{\Arc\in\cycleo}\left(\card{\eincycles{\Arc}}\right)\). Consider the vector 
\begin{equation}\label{eq: prop ff eq 1}
\circuit \coloneqq \sum_{\cycle\in\cyclebasis\setminus\{\cycleo\}}\cycle-\card{\eincycles{\Arco}}\cycleo \in \Dcyclespace.
\end{equation} 
We prove that $\circuit$ is non-negative. 
For every $\Arc \notin \cycleo$, we clearly have $\circuit(\Arc) \geq0$. 
Suppose now that $\Arc \in \cycleo$. Then
\[
\circuit(\Arc) = \sum_{\cycle\in\eincycles{\Arc}}\cycle(\Arc) - \card{\eincycles{\Arco}}\cycleo(\Arc) = \card{\eincycles{\Arc}} - \card{\eincycles{\Arco}}\cycleo(\Arc) \geq 0
\]
since \(\card{\eincycles{\Arco}}\) was chosen to be minimum and $\cycleo(\Arc) \in \{0, 1\}$ as $\cyclebasis$ is forward. Then, by Lemma~\ref{lemma:generators of forward cycles}, there exists \(\lambda_{\cycle} \geq 0\) for all \(\cycle\in\forwardcycles(\digraph)\) such that \begin{equation}\label{eq: prop ff eq 2}
\circuit = \sum_{\cycle\in\forwardcycles(\digraph)}\lambda_{\cycle}\cycle.
\end{equation} 
Since \(\forwardcycles(\digraph) = \cyclebasis\), and $\cyclebasis$ is a basis of \(\Dcyclespace\), the linear combinations in equation \eqref{eq: prop ff eq 1} and \eqref{eq: prop ff eq 2} have the same coefficients. In particular:
\[
-\card{\eincycles{\Arco}} = \lambda_{\cycleo} \geq 0, 
\]
which implies \(\card{\eincycles{\Arco}}=0\) as required.
\end{proof}

Theorem~\ref{thm: forward fundamental existence} gives a polynomial-time algorithm at hand to decide whether a digraph $\digraph$ admits a forward fundamental cycle basis: Check at first whether all blocks of $\digraph$ are strongly connected or a single arc. If not, then $\digraph$ cannot have a forward cycle basis by Proposition~\ref{prop: weakly fundamental forward cycle basis}. Otherwise, compute a forward weakly fundamental cycle basis $\cyclebasis$ using a directed ear decomposition. We may then verify that for all \(\cycleo \in \cyclebasis\), there exists \(\Arc\in \cycleo\) such \(\Arc \notin \bigcup_{\cycle \in \cyclebasis\setminus \{\cycleo\}}\cycle\) to check whether $\cyclebasis$ is also strictly fundamental. If yes, then we have found a forward fundamental cycle basis. If no, then $\digraph$ cannot have a forward fundamental cycle basis, as the existence of such a basis would contradict the uniqueness of a forward cycle basis by Theorem~\ref{thm: forward fundamental existence}.

\section{APX-hardness of the Minimum-Weight Forward Weakly Fundamental Cycle Basis Problem}
\label{sec:APX-hardness}

In this section, we consider the problem of finding forward cycle bases with a minimum weight. More precisely, let \(\digraph\) be a digraph endowed with arc weights \(\weight:\Darcs \to \mathbb{R}_{\geq0}\). For each cycle $\cycle$ in $\digraph$, we define its weight as
\[\weight(\cycle) \coloneqq \sum_{\Arc \in \cycle} \weight(\Arc),\]
and the total weight of a cycle basis \(\cyclebasis\) is given by 
\[\weight(\cyclebasis) \coloneqq \sum_{\cycle \in \cyclebasis} \weight(\cycle).\]

\begin{definition}
    Given a digraph $\digraph$ and arc weights $\weight:\Darcs \to \mathbb{R}_{\geq0}$, the \emph{minimum-weight forward (weakly fundamental, fundamental) cycle basis problem} is to find a forward (weakly fundamental, fundamental) cycle basis of $\cyclebasis$ such that $\weight(\cyclebasis)$ is minimum or to decide that no forward cycle basis exists.
\end{definition}

  While existence and complexity results of minimum not necessarily forward cycle bases are known, not much research has been done for their forward counterparts. Table \ref{tab:complexity} summarizes known complexity results for minimum cycle basis problems.

\begin{table}[htbp]
\centering
\begin{threeparttable}
\renewcommand{\arraystretch}{1.25}
\caption{Known computational complexity of the MCB problem for different classes of cycle bases.}
\label{tab:complexity}
\begin{tabular}{@{}lll@{}}
\toprule
\textbf{Type} & \textbf{Non-forward} & \textbf{Forward} \\
\midrule
Directed      & P \cite{horton1987} & P \cite{forwardGleiss} \\
Undirected      & P \cite{horton1987} & P \cite{forwardGleiss} \\
Integral      & ?      & ? \\
Weakly fundamental  & APX-hard \cite{Rizzi2009}       & \textbf{APX-hard}\tnote{*} \\
Fundamental         & APX-hard  \cite{galbiati2003approximability}       & \textbf{P}\tnote{*} \\
\bottomrule
\end{tabular}
\begin{tablenotes}
    \item[*] New result.
    \item[?] Unknown complexity status.
\end{tablenotes}
\end{threeparttable}
\end{table}

The existence of forward (weakly fundamental) cycle bases has already been characterized by Proposition~\ref{prop: weakly fundamental forward cycle basis}. The problem is polynomial-time solvable for arbitrary forward cycle bases \cite{forwardGleiss}. For forward fundamental cycle bases, we can use Theorem~\ref{thm: forward fundamental existence}.

\begin{corollary}
A minimum-weight forward fundamental cycle basis of $\digraph$, if it exists, 
can be computed in polynomial time.  
If no such basis exists, its nonexistence can also be certified in polynomial time.
\end{corollary}
\begin{proof}
    If a forward fundamental cycle basis exists, then it is unique by Theorem~\ref{thm: forward fundamental existence}, and hence trivially optimal. We can decide the existence as discussed at the end of Section~\ref{sec:existence}.
\end{proof}

The goal of the remainder of this section is to show that the minimum-weight forward weakly fundamental cycle basis problem is APX-hard. 
We give an L-reduction from the non-forward version on digraphs with metric weights, whose APX-hardness -- even for uniform weights -- has been established by Rizzi \cite{Rizzi2009}.

Let $\digraph$ be a digraph and let further $\weight:\Darcs \to \mathbb{R}_{\geq0}$ be \emph{metric} arc weights, i.e., for any arc $\Arc = (u, v)$, $\weight(\Arc)$ is the cost of a shortest $u$-$v$-path w.r.t.\ $\weight$.
We construct a directed graph $\Rdigraph$ by adding to \(\Darcs\), for each $\Arc= (u,v)\in \Darcs$, the arc $\Rarc \coloneqq(v,u)$,  with weight equal to $\weight((u,v))$. 
This transformation clearly produces a directed graph of size polynomial in $\card{\Dvertices}+\card{\Darcs}$, where each weakly connected component is strongly connected. 
Define
\[
\MWFCB{\digraph} \coloneqq
\min \left\{ \weight(\cyclebasis) \;\middle|\; \begin{array}{c} \cyclebasis \text{ is a weakly fundamental} \\ \text{ cycle basis of } \digraph \end{array} \right\}
\]
and
\[
\MWFFCB{\Rdigraph} \coloneqq
\min \left\{ \weight(\cyclebasis) \;\middle|\;
\begin{array}{c}
  \cyclebasis \text{ is a weakly fundamental} \\
  \text{forward cycle basis of } \Rdigraph
\end{array}
\right\}.
\]
For the sake of constructing an L-reduction we need to prove that there exist \(\alpha,\beta \geq 0\) such that:
\begin{enumerate}
    \item For all digraphs \(\digraph\), we have 
    \begin{equation}\label{eq: alpha bound}
        \MWFFCB{\Rdigraph} \leq \alpha \MWFCB{\digraph}
    \end{equation}
    \item For any weakly fundamental forward cycle basis \(\cyclebasis'\) of \(\Rdigraph\), we can find a weakly fundamental cycle basis \(\cyclebasis\) of \(\digraph\) in polynomial time such that:
    \begin{equation}\label{eq: beta bound}
        |\weight(\cyclebasis)- \MWFCB{\digraph}| \leq \beta |\weight(\Rcyclebasis) - \MWFFCB{\Rdigraph}|
    \end{equation}
\end{enumerate}

\begin{observation}\label{obs:two-way-cycle-rank}
    The cycle rank of \(\Rdigraph\) is equal to 
    \[\Rcyclerank \coloneqq 2\card{\Darcs} - \card{\Dvertices} + 1.\]
\end{observation}

In particular, the cycle rank of \(\Rdigraph\) exceeds that of \(\digraph\) by exactly \(\card{\Darcs}\).
We now focus on a distinguished family of directed cycles in \(\Rdigraph\).
For each arc \(\Arc = (\tail,\head) \in \Darcs\), let \(\simplecycle{\Arc}\) denote the directed \(2\)-cycle in $\Rdigraph$ formed by the arcs \((\tail,\head)\) and \((\head,\tail)\). Since \(\weight\) is metric, \(\simplecycle{\Arc}\) is a smallest weight cycle in \(\Rdigraph\) containing \(\Arc\). These cycles constitute precisely the set
\[
\simplecycles \coloneqq \{\simplecycle{\Arc} \mid \Arc \in \Darcs\}
\]
of all \(2\)-cycles of \(\Rdigraph\).

Furthermore, for every cycle \(\cycle\) in \(\digraph\), we define a directed counterpart \(\Rcycle\) by choosing, for each arc \(\Arc=(\tail,\head)\in\cycle\), one of the two orientations \((\tail,\head)\) or \((\head,\tail)\) such that \(\Rcycle\) is a directed cycle in \(\Rdigraph\).
We call \(\Rcycle\) a \emph{directed lift} of \(\cycle\) in \(\Rdigraph\).
There are exactly two such directed lifts, corresponding to the two traversal directions of \(\cycle\); either choice is admissible, and we fix one.
In particular, for every \(\Arc \in \cycle\), there exists \(\ArcC \in \{\Arc,\Rarc\}\) such that \(\ArcC \in \Rcycle\). See Figure \ref{fig: reduction} for an illustrative example.

Finally, there is a natural extension of any cycle basis \(\cyclebasis\) of \(\digraph\) to a \emph{forward} cycle basis of \(\Rdigraph\), obtained by augmenting the set
\(\{\Rcycle \mid \cycle \in \cyclebasis\}\) with the \(2\)-cycles \(\simplecycle{\Arc}\):%

\begin{figure}[h!]
    \centering
\begin{minipage}{0.45\textwidth}
\begin{tikzpicture}[
  every node/.style={circle, draw, minimum size=6mm, inner sep=1pt, fill=white},
  >={Stealth[length=6pt]}
]

  \def\scaleconst{3}
  \def\shiftconst{3}
  \coordinate (A1) at ({0*\scaleconst},{0*\scaleconst});
  \coordinate (A2) at ({1*\scaleconst},{0.75*\scaleconst});
  \coordinate (A3) at ({2*\scaleconst},{0*\scaleconst});
  \coordinate (A4) at ({1*\scaleconst},{-0.75*\scaleconst});

  \coordinate (B1) at ({0.5*\scaleconst},{0*\scaleconst});
  \coordinate (B2) at ({1*\scaleconst},{0.25*\scaleconst});
  \coordinate (B3) at ({1.5*\scaleconst},{0*\scaleconst});
  \coordinate (B4) at ({1*\scaleconst},{-0.25*\scaleconst});

  \newcommand{\uarc}[2]{%
    \draw[->, shorten >=8pt, shorten <=8pt] (#2) to (#1);
  }

  {\color{skyblue}\uarc{A1}{A2}}
  {\color{black}\uarc{A2}{A3}}
  {\color{vividorchid}\uarc{A3}{A4}}
  \uarc{A4}{A1}

  {\color{skyblue}\uarc{B1}{B2}}
  {\color{black}\uarc{B2}{B3}}
  \uarc{B3}{B4}
  \uarc{B4}{B1}

  {\color{skyblue}\uarc{A1}{B1}}
  {\color{skyblue}\uarc{A2}{B2}}
  \uarc{A3}{B3}
  \uarc{A4}{B4}

  \node[draw=skyblue,text=skyblue] (n1) at (A1) {1};
  \node[draw=skyblue,text=skyblue] (n2) at (A2) {2};
  \node[draw=vividorchid] (n3) at (A3) {{\color{vividorchid}3}};
  \node[draw=vividorchid] (n4) at (A4) {{\color{vividorchid}4}};

  \node[draw=skyblue,text=skyblue] (n5) at (B1) {5};
  \node[draw=skyblue,text=skyblue] (n6) at (B2) {6};
  \node (n7) at (B3) {7};
  \node (n8) at (B4) {8};  
  \node[shape=rectangle, draw=none, inner sep=0pt, minimum size=0pt, right=of A4, yshift = 33pt, text=vividorchid] (n9) {\(\Arc\)};
  \node[shape=rectangle, draw=none, inner sep=0pt, minimum size=0pt, right=of A1, yshift = 40pt, xshift =  -10pt, text=skyblue] (n9) {\(\cycle\)};

\end{tikzpicture}
\end{minipage}

\vspace{0.5cm}
\begin{minipage}{0.45\textwidth}
\begin{tikzpicture}[
  every node/.style={circle, draw, minimum size=6mm, inner sep=1pt, fill=white},
  >={Stealth[length=6pt]}
]

  \def\scaleconst{3}
  \def\shiftconst{3}
  \coordinate (A1) at ({0*\scaleconst},{0*\scaleconst});
  \coordinate (A2) at ({1*\scaleconst},{0.75*\scaleconst});
  \coordinate (A3) at ({2*\scaleconst},{0*\scaleconst});
  \coordinate (A4) at ({1*\scaleconst},{-0.75*\scaleconst});

  \coordinate (B1) at ({0.5*\scaleconst},{0*\scaleconst});
  \coordinate (B2) at ({1*\scaleconst},{0.25*\scaleconst});
  \coordinate (B3) at ({1.5*\scaleconst},{0*\scaleconst});
  \coordinate (B4) at ({1*\scaleconst},{-0.25*\scaleconst});

  \newcommand{\doubledarc}[2]{%
    \draw[->, shorten >=8pt, shorten <=8pt, bend left=10] (#2) to (#1);
    \draw[->, shorten >=8pt, shorten <=8pt, bend left=10] (#1) to (#2);
  }
  \newcommand{\doubledarccolor}[4]{%
    \draw[->, shorten >=8pt, shorten <=8pt, bend left=10, draw=#4] (#2) to (#1);
    \draw[->, shorten >=8pt, shorten <=8pt, bend left=10, draw=#3] (#1) to (#2);
  }

  \doubledarccolor{A1}{A2}{black}{skyblue}
  \doubledarc{A2}{A3}
  {\color{vividorchid}
  \doubledarc{A3}{A4}}
  \doubledarc{A4}{A1}

  \doubledarccolor{B1}{B2}{skyblue}{black}
  \doubledarc{B2}{B3}
  \doubledarc{B3}{B4}
  \doubledarc{B4}{B1}

  \doubledarccolor{A1}{B1}{skyblue}{black}
  \doubledarccolor{A2}{B2}{black}{skyblue}
  \doubledarc{A3}{B3}
  \doubledarc{A4}{B4}

  \node[draw=skyblue,text=skyblue] (n1) at (A1) {1};
  \node[draw=skyblue,text=skyblue] (n2) at (A2) {2};
  \node[draw=vividorchid] (n3) at (A3) {{\color{vividorchid}3}};
  \node[draw=vividorchid] (n4) at (A4) {{\color{vividorchid}4}};

  \node[draw=skyblue,text=skyblue] (n5) at (B1) {5};
  \node[draw=skyblue,text=skyblue] (n6) at (B2) {6};
  \node (n7) at (B3) {7};
  \node (n8) at (B4) {8};
   \node[shape=rectangle, draw=none, inner sep=0pt, minimum size=0pt, right=of A4, yshift = 40pt, xshift =  -10pt, text=vividorchid] (n9) {\(\simplecycle{\Arc}\)};
   \node[shape=rectangle, draw=none, inner sep=0pt, minimum size=0pt, right=of A1, yshift = 40pt, xshift =  -10pt, text=skyblue] (n9) {\(\Rcycle\)};

\end{tikzpicture}
\end{minipage}
    \caption{Reduction applied to the Wagner graph with highlighted 2-cycle, and cycle lifting.}\label{fig: reduction}
\end{figure}
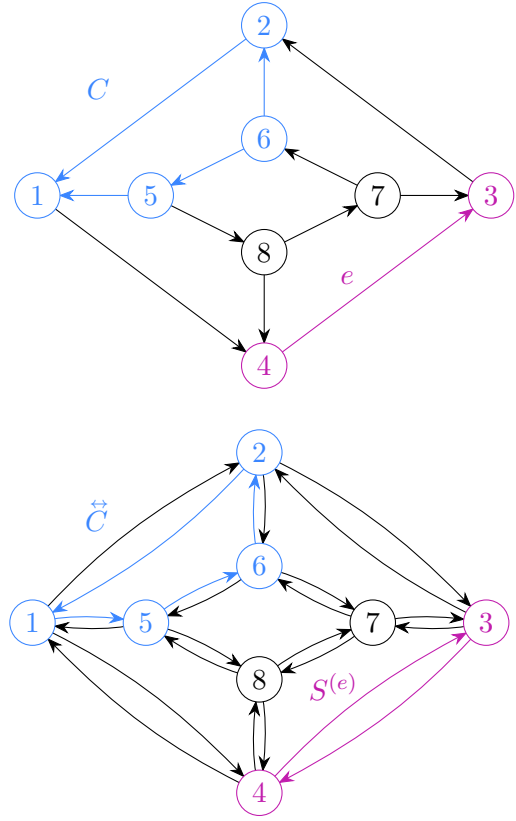

\begin{lemma}[Forward lifting cycle bases]\label{lem:cycle-basis-lifting}
    If \cyclebasis \ is a cycle basis of \digraph, then 
    \begin{equation}
        \Rcyclebasis \coloneqq \{\Rcycle \mid \cycle \in \cyclebasis \} \cup \simplecycles
    \end{equation}
    is a forward cycle basis of \(\Rdigraph\).
\end{lemma}
\begin{proof}
    Since  \(|\Rcyclebasis| = 2\card{\Darcs} - \card{\Dvertices} + 1 = \Rcyclerank\) (Observation~\ref{obs:two-way-cycle-rank}), it suffices to show that \(\Rcyclebasis\) spans the cycle space \(\cyclespace{\Rdigraph}{\bR}\).
    Observe that for every cycle \(\cycle\) in $\Rdigraph$,  the vector 
    \begin{equation}\label{eq: switch}
        \cycle' \coloneqq \cycle - \sum_{\substack{\Arc\in \Darcs \\ \text{s.t. } \cycle(\Rarc)\neq 0}} \simplecycle{\Arc}
    \end{equation}
    has support in \(\Darcs\), as this corresponds to substituting arc \(\Rarc\) in \(\cycle\) with \(\Arc\) whenever \(\Rarc\) is not in \(\Darcs\) by subtracting the cycle \(\simplecycle{\Arc}\) from \(\cycle\). 
    In particular, $\cycle' \in \Span(\cyclebasis)$, so that $\cycle \in \Span(\cyclebasis \cup \simplecycles)$. It now suffices to show that each cycle $\cycle \in \cyclebasis$ lies in $\Span(\Rcyclebasis \cup \simplecycles)$. However, in this case $\bigl(\Rcycle\bigr)'
 = \cycle$, and \eqref{eq: switch} realizes $C$ as a linear combination of $\Rcycle \in \Rcyclebasis$ and elements of $\simplecycles$.
\end{proof}

\begin{lemma}[Weakly fundamental lifting and weights]\label{lem:lifting-cost}
Let \(\cyclebasis\) be any weakly fundamental cycle basis of \(\digraph\). Then $\Rcyclebasis$ is a forward weakly fundamental cycle basis of \(\Rdigraph\) and 
\begin{equation}\label{eq:lifting-weight}
\weight(\Rcyclebasis) \;=\; \weight(\cyclebasis) \;+\;2\weight(\Darcs).
\end{equation}
In particular, lifting increases the objective by an additive constant independent of the choice of \(\cyclebasis\) and
\begin{equation}\label{eq:lifting-opt}
\MWFFCB{\Rdigraph} \leq \MWFCB{\digraph} +2\weight(\Darcs).
\end{equation}
\end{lemma}

\begin{proof}
Since $B$ is weakly fundamental, there exists an ordering of the cycles \(\cyclebasis=\{\cycle_i\}_{i=1}^{\Dcyclerank}\), 
where \(\Dcyclerank \coloneqq \cyclerank{\digraph}\), such that for every \(1<k\leq \Dcyclerank\) 
there is an arc \(\Arc_k\in \cycle_k\) with \(\Arc_k\notin \bigcup_{i<k}\cycle_i\).  
Then also \(\RarcC{\Arc_k}{\cycle_k}\in\Rcycle_k\) and 
\(\RarcC{\Arc_k}{\cycle_k}\notin\bigcup_{i<k}\Rcycle_i\). 

An ordering that makes \(\Rcyclebasis\) weakly fundamental is therefore obtained by first listing 
the cycles in \(\simplecycles \setminus\{\simplecycle{\Arc_k}\mid k=1,\ldots,\Dcyclerank\}\), 
then \(\Rcycle_1\), and subsequently, for each \(k=1,\ldots,\Dcyclerank\), 
placing \(\Rcycle_k\) immediately before \(\simplecycle{\Arc_k}\).  
Since exactly one of \(\Arc_k\) or \(\Rop{\Arc_k}\) belongs to \(\Rcycle_k\), 
while both are contained in \(\simplecycle{\Arc_k}\), 
this ordering ensures that \(\Rcyclebasis\) is indeed weakly fundamental.

Verifying \eqref{eq:lifting-weight} and \eqref{eq:lifting-opt} is straightforward.
\end{proof}

To complete the reduction, we must prove the converse:
There is an optimal forward weakly fundamental cycle basis of \(\Rdigraph\) that is obtained as a lifting of an optimal weakly fundamental cycle basis of \(\digraph\). First, we prove that we can easily construct an optimal forward weakly fundamental cycle basis of \(\digraph\) that contains the cycles in \(\simplecycles\):

\begin{lemma}[Containment of \(\simplecycles\) in forward cycle bases]\label{lem:include-2cycles}
For every forward weakly fundamental cycle basis \(\cB\), we can find in polynomial time a forward weakly fundamental cycle basis \(\cyclebasis^*\) of \(\Rdigraph\)  such that \(\simplecycles \subset \cyclebasis^*\) and \(\weight(\cyclebasis^*) \leq \weight(\cyclebasis)\).
\end{lemma}

\begin{proof}
Let \(\cyclebasis\) be a forward weakly fundamental cycle basis of \(\Rdigraph\).  
We order the basis \(\cyclebasis=\{\cycle_k\}_{k=1}^{\Rcyclerank}\) such that for every index \(1<k\leq \Rcyclerank\)
there exists an arc \(\Arc_k\in\cycle_k\) with
\(\Arc_k\notin\bigcup_{i<k}\cycle_i\). Removing the chord set \(\{\Arc_k\}_{k=1}^{\Rcyclerank}\)
from \(\Rdigraph\) yields a spanning forest \(\tree\): If we remove $\Arc_{\Rcyclerank}$ from $\Rdigraph$, all weakly connected components remain connected, the cycle rank of the resulting graph drops by exactly one, and $\cyclebasis \setminus \{\cycle_\Rcyclerank\}$ is weakly fundamental cycle basis. Iterating this process, we ultimately arrive at a spanning forest \cite{Rizzi2009}.

By construction each cycle \(\cycle_k\) is
contained in \(\tree\cup\{\Arc_i\}_{i\le k} \). 
Let
\[
\minimalcyclesk \coloneqq \{\cycleo \in \forwardcycles(\Rdigraph) \mid \cycleo \subset \tree\cup\{\Arc_i\}_{i\le k} \text{ and } \Arc_k\in\cycleo\}
\]
be the family of all directed cycles that lie inside  \(\tree\cup\{\Arc_i\mid i\le k\}\) and contain \(\Arc_k\).
A forward weakly fundamental basis \(B^*\) can be obtained by choosing, for each \(k\), any directed cycle \(\cycleo_k\) of minimum weight in \(\minimalcyclesk\). In particular, \(\weight(\cycleo_k) \leq \weight(\cycle_k)\), and hence \(\weight(\cyclebasis^*) \leq \weight(\cyclebasis)\). \(\cycleo_k\) can be found by concatenating \(\Arc_k\) with a shortest path in \(\tree \cup {\{\Arc_i\}}_{i \leq k}\) from the head to the tail of \(\Arc_k\).

We show that we can choose \(\cycleo_k\) so that \(\simplecycles \subset \cyclebasis^*\). For each \(\Arc \in \Darcs\) we proceed as follows: Consider the corresponding \(2\)-cycle \(\simplecycle{\Arc}\) in \(\Rdigraph\).
Since \(\tree\) is a spanning tree of \(\Rdigraph\), it cannot contain both \(\Arc\) and its reverse \(\Rop{\Arc}\).
Consequently, at least one of these two opposite arcs coincides with a chord \(\Arc_j\).
Let \(k \geq j\) be the largest index such that \(\Arc_k\) is one of the two arcs of \(\simplecycle{\Arc}\).

By construction, we have
\[
\simplecycle{\Arc} \subseteq \tree \cup \{\Arc_i\}_{i \le k}
\quad\text{and}\quad
\Arc_k \in \simplecycle{\Arc},
\]
which implies \(\simplecycle{\Arc} \in \minimalcyclesk\).
Moreover, \(\simplecycle{\Arc}\) is a minimum-weight cycle in \(\minimalcyclesk\), and we can pick \(\cycleo_k = \simplecycle{\Arc} \). The resulting basis \(\cyclebasis^*\) is a forward weakly fundamental cycle basis satisfying
\(\simplecycles \subset \cyclebasis^*\), and  as claimed.
\end{proof}

\begin{lemma}[Projection to the original problem]\label{lem:projection} Let \(\cyclebasis^*\) be a forward weakly fundamental basis of \(\Rdigraph\) that contains  \(\simplecycles\). Then, removing the family \(\simplecycles\) from \(\cyclebasis^*\) and projecting each remaining directed cycle to its underlying cycle in \(\digraph\) yields a weakly fundamental cycle basis \(\cyclebasis\) of \(\digraph\). Moreover,
\begin{equation}\label{eq: cost identity}
   \weight(\cyclebasis^*) = \weight(\cyclebasis) +2\weight(\Darcs), 
\end{equation}
Furthermore, if \(\cyclebasis^*\) has minimum weight, so does \(\cyclebasis\).
 \end{lemma}
\begin{proof}

Consider
\(
\widetilde{\cB} \coloneqq \cyclebasis^* \setminus \simplecycles
\)
and define for each \(\widetilde\cycle\in\widetilde{\cB}\) the projection \(\proj(\widetilde\cycle) \coloneqq \cycle\in\cyclespace{\digraph}{\bR}\) by setting, for every \(\Arc\in\Darcs\),
\[
\cycle(\Arc)\coloneqq\widetilde\cycle(\Arc)\;-\;\widetilde\cycle(\Rop{\Arc}).
\]

Because \(\cyclebasis^*\) is weakly fundamental in \(\Rdigraph\) there exists an ordering of its cycles \(\{\widetilde \cycle_k\}_{k=1}^{\widetilde\mu}\) in which each cycle contains an arc $\widetilde\Arc_k$ not used by any earlier cycle. Defining $\cycle_i \coloneqq \proj(\widetilde\cycle_i)$, this induces an ordering of \(\cyclebasis \coloneqq \{\cycle_i\}_{i=1}^{\Dcyclerank}\).

We claim that for every \(k\), the arc \(\Arc_k \in \{\widetilde{\Arc}_k,\overset{\leftrightarrow}{\widetilde{\Arc}_k}\}\) with \(\Arc_k \in \cycle_k\) does not belong to any cycle \(\cycle_i\) with \(i<k\). Indeed, by construction, \(\widetilde{\Arc}_k\) is not contained in any cycle preceding \(\widetilde{\cycle}_k\) in the original ordering of \(\cyclebasis^*\); hence no earlier \(\widetilde{\cycle}_i\) contains the same directed arc. In particular, the cycle \(\simplecycle{\Arc_k}\) appears after \(\widetilde{\cycle}_k\) in the ordering. But then also any cycle in $\cyclebasis^*$ containing $\overset{\leftrightarrow}{\widetilde{\Arc}_k}$ can only appear later than $\widetilde\cycle_k$. 
We conclude that $\cyclebasis$ is weakly fundamental in \(\digraph\), as required.

Equation \eqref{eq: cost identity} follows from the fact that the weights of the cycles in \(\cyclebasis\) are equal to the weights of the cycles in \(\widetilde\cyclebasis\) and that \(\weight(\simplecycles)= 2\weight(\Darcs)\). Then \eqref{eq: cost identity}  implies that 
\begin{equation}
    \MWFFCB{\Rdigraph} \geq \MWFCB{\digraph} +2\weight(\Darcs)
\end{equation}
and together with Lemma \ref{lem:lifting-cost}, this implies that the inequality is actually an equality. Therefore, if \(\cyclebasis^*\) is minimum, then \(\cyclebasis\) is a minimum-weight weakly fundamental cycle basis of \(\digraph\).
\end{proof}
\begin{theorem}[APX-hardness]\label{thm:APX-hard}
The minimum-weight forward weakly fundamental cycle basis problem is APX-hard. The APX-hardness persists when the weights are metric or uniform.
\end{theorem}

\begin{proof}
Given an instance \((\digraph,\weight)\) of the minimum-weight weakly fundamental cycle basis problem with metric/uniform weights, we construct \((\Rdigraph,\widetilde\weight)\) with metric/uniform weights in polynomial time. Due to Observation~\ref{obs:blocks}, we can assume that $\digraph$ is 2-connected. By Lemma~\ref{lem:projection}, we have the following cost relation:
\begin{equation}\label{eq: obj relation}
    \MWFFCB{\Rdigraph} = \MWFCB{\digraph} + 2\weight(\Darcs).
\end{equation}
Since every arc in a 2-connected graph is contained in at least one cycle of a cycle basis, we have \(2\weight(\Darcs) \leq 2\MWFCB{\digraph}\). Hence, for $\alpha = 3$, we find \(\MWFFCB{\Rdigraph} \leq \alpha\MWFCB{\digraph}\).

Furthermore, given a forward weakly fundamental cycle basis \(\Rcyclebasis\) of \(\Rdigraph\), by Lemma~\ref{lem:include-2cycles} we can find a forward weakly fundamental cycle basis \(\cyclebasis^*\) of smaller or equal weight containing \(\simplecycles\). By Lemma~\ref{lem:projection}, we can then obtain a weakly fundamental cycle basis \(\cyclebasis\) of \(\digraph\) in polynomial time such that
\(\weight(\cyclebasis) = \weight(\cyclebasis^*) - 2\weight(\Darcs) \leq \weight(\Rcyclebasis) - 2\weight(\Darcs)\). Subtracting \(\MWFCB{\digraph}\) from both sides and substituting Equation~\eqref{eq: obj relation}, it directly follows that:
\[ \weight(\cyclebasis) - \MWFCB{\digraph} \leq \weight(\Rcyclebasis) - \MWFFCB{\Rdigraph}. \] Thus Equation~\eqref{eq: beta bound} holds with \(\beta = 1\), and our reduction is indeed an L-reduction of the minimum-weight forward weakly fundamental cycle basis problem from the minimum weakly fundamental cycle basis problem, which is APX-hard for uniform weights on 2-connected graphs \cite{Rizzi2009}.
\end{proof}

For completeness, we finally show that the assumption that \((\digraph,\weight)\) has metric weights is necessary for our reduction to work. We present a counterexample for which Lemma~\ref{lem:include-2cycles} and Lemma~\ref{lem:projection} do not hold.
\begin{example}
Consider the digraph $\digraph$ with \(\Dvertices = \{a,b,c\}\) and \(\Darcs = \{(a,b),(b,c),(c,a)\}\) with weights 
\(\weight((a,b)) = \weight((b,c)) \coloneqq 1\) and \(\weight((c,a)) = 3\).  
In \(\digraph\) there exists a unique cycle \(\cycle\); in particular,  
\(\cyclebasis \coloneqq \{\cycle\}\) forms the minimum-weight weakly fundamental cycle basis of \(\digraph\).

However, \(\Rcyclebasis = \{\Rcycle\} \cup \simplecycles\) is not a minimum-weight forward weakly fundamental cycle basis of \(\Rdigraph\), since it has weight \(15\), while the forward weakly fundamental cycle basis 
\begin{align*}
\cyclebasis' &\coloneqq \{\simplecycle{(a,b)}, \simplecycle{(b,c)}, \cycle', \cycle''\}, \quad \text{ where }\\
\cycle' &= a,(a,b),b,(b,c),c,(c,a),a \quad \text{ and }\\
\cycle'' &= a,(a,c),c,(c,b),b,(b,a),a    
\end{align*}
is optimal and has weight \(14\).
Thus, Lemma~\ref{lem:projection} does not hold for non-metric weights, as \(\Rcyclebasis\) is not an optimal forward weakly fundamental cycle basis of \(\Rdigraph\). Moreover, Lemma~\ref{lem:include-2cycles} does not hold, as any forward cycle basis of $\Rdigraph$ containing all three cycles in $\simplecycles$ must have weight at least 15.

\end{example}

\section{Integrality of Minimum Cycle Bases}\label{sec:integrality of minimum cycle bases}

While the exact computational complexity of the minimum integral cycle basis problem remains an open question for general graphs, it is polynomial-time solvable for planar graphs. In particular, any algorithm that computes a minimum cycle basis over the subset of isometric cycles yields an integral cycle basis, since isometric cycle bases are always integral in planar graphs~\cite{Rizzi2009}. This raises a natural structural question: under what general conditions does a minimum directed cycle basis algorithm inherently guarantee integrality?

More formally, we address the following problem: for which graphs does a minimum cycle basis algorithm yield an integral cycle basis under \emph{any} possible weight assignment?
\begin{definition}[\mcbeq/\notmcbeq]\label{def:opt-in}
We say that a graph is \emph{\mcbeq}\ (\text{OPTimal INteger}) if, for every edge weight assignment, the weight of a minimum directed cycle basis equals the weight of a minimum integral cycle basis. Conversely, a graph that is not \mcbeq\ is called \emph{\notmcbeq}.
\end{definition}

\begin{observation}\label{obs:undirected-minors}
    The \mcbeq{}/\notmcbeq{} notion does not depend on the orientation of the arcs, so that this in fact a property of the underlying undirected graph: The cycle spaces of two different orientations are isomorphic, and the cycle basis determinant remains the same up to sign \cite{cyclebasisintro}
\end{observation}
While Definition~\ref{def:opt-in} might initially appear slightly decoupled from the algorithmic question posed at the beginning of this section, the two are closely related: if a minimum directed cycle basis algorithm always returns an integral cycle basis on a given graph, then the weight of a minimum directed cycle basis and the weight of a minimum integral cycle basis necessarily coincide, so the graph is \mcbeq. The converse that being \mcbeq\ in turn forces such algorithmic behavior is established by the following proposition.
\begin{proposition}\label{prop: MICBP in optin is in P}
Let \(\digraph\) be an \mcbeq\ graph. Any greedy minimum cycle basis algorithm
(e.g.\ Horton's or de Pina's algorithm) that breaks weight ties according to a
fixed lexicographic order on the arcs returns an integral cycle basis. In
particular, the minimum integral cycle basis problem on \mcbeq\ graphs can be
solved in polynomial time by such an algorithm.
\end{proposition}

\begin{proof}

Let \(\digraph\) be an \mcbeq\ graph. We begin by perturbing the arc weights so that the weight of a cycle uniquely determines the set of arcs composing it.

Let \(\weight \in \mathbb{Q}^{\Arcs{\digraph}}\) be a finite-length weight function, and fix an ordering \(\Arcs{\digraph}=\{\Arc_1,\ldots,\Arc_m\}\). For each arc \(\Arc_i\), define the perturbed weight

\[
\tweight(\Arc_i) \coloneqq \weight(\Arc_i) + \varepsilon_i,
\qquad
\varepsilon_i \coloneqq 10^{-k-m+i},
\]

where \(k-1\) is the smallest (most negative) decimal exponent appearing in the representation of \(\weight\).

With this choice, the weight of any cycle \(\cycle\) uniquely encodes its arc incidence: the digit in position \(-k-m+i\) is equal to \(1\) if and only if \(\Arc_i \in \cycle\). Hence, \(\tweight(\cycle)\) uniquely identifies \(\cycle\).

Consider the perturbed minimum directed cycle basis problem \((\digraph,\tweight)\). The minimum cycle basis is unique. Indeed, any optimal cycle basis has the same nondecreasing sequence of cycle weights \cite{cyclebasisintro}, and under the perturbation each weight corresponds to a unique cycle. Therefore, the weight sequence identifies a unique cycle basis \(\cyclebasis\). Since \(\digraph\) is \mcbeq, \(\cyclebasis\) must be an integral cycle basis.

We now run a minimum directed cycle basis algorithm on \((\digraph,\weight)\), breaking ties lexicographically according to the vectors
\[
\bigl(\weight(\cycle), (\cycle_i)_{i=1}^m\bigr),
\]
where \(\cycle_i \in \{0, 1\}\) denotes the incidence of \(\Arc_i\) in \(\cycle\), as described in \cite{cyclebasisintro}. We claim that comparing two cycles \(\cycle,\cycle'\) of \(\digraph\) under \(\tweight\) is equivalent to comparing them under \(\weight\) with ties broken lexicographically by incidence vector. Indeed, if \(\weight(\cycle)\neq\weight(\cycle')\), then since \(\sum_{i=1}^m \varepsilon_i < 10^{-k+1}\), the perturbation cannot exceed the gap \(|\weight(\cycle)-\weight(\cycle')| \geq10^{-k+1}\), so \(\tweight(\cycle)<\tweight(\cycle')\) if and only if \(\weight(\cycle)<\weight(\cycle')\). If instead \(\weight(\cycle)=\weight(\cycle')\), then \(\tweight(\cycle)-\tweight(\cycle') = \sum_{i:\,\Arc_i\in\cycle}\varepsilon_i - \sum_{i:\,\Arc_i\in\cycle'}\varepsilon_i\), whose sign is determined by the most significant index \(i\) at which \(\cycle\) and \(\cycle'\) disagree, which is exactly the lexicographic comparison of their incidence vectors.

The resulting cycle basis $\cyclebasis$ is then also optimal for \((\digraph,\tweight)\). By the above, $\cyclebasis$ is integral.
\end{proof}

All planar graphs are \mcbeq. In fact, \mcbeq\ graphs share another important property with planar graphs: the class is minor-closed.

\begin{theorem}\label{thm: optin is minor closed}
If a graph \(\digraph\) is \mcbeq, then any minor of \(\digraph\) is also \mcbeq.
\end{theorem}

\begin{proof}
It suffices to show that if a \notmcbeq\ graph is obtained from a graph
\(\digraph\) by edge deletion, vertex deletion, or edge contraction, then
\(\digraph\) is also \notmcbeq.

We first consider edge deletion. Let \(\Arco \in \Darcs\) and define
\(\odigraph \coloneqq \digraph \setminus \Arco\).
If \(\odigraph\) is \notmcbeq, there exists a weight assignment
\(\oweight \in \bR^{\Arcs{\odigraph}}\) such that no directed minimum cycle basis of
\(\odigraph\) is integral.
Define a weight assignment \(\weight\) on \(\digraph\) by setting
\(\weight(\Arc) \coloneqq \oweight(\Arc)\) for all
\(\Arc \in \Arcs{\odigraph}\), and choosing \(\weight(\Arco)\) to be strictly
larger than the maximum weight of any cycle appearing in a minimum cycle basis
of \(\odigraph\).

With this choice of weights, the greedy algorithm for computing a minimum cycle
basis of \(\digraph\) necessarily selects a cycle basis entirely contained in
\(\odigraph\), since any cycle using \(\Arco\) is too expensive.
After these cycles are chosen, the algorithm completes the basis by adding a
minimum-weight cycle containing \(\Arco\).
Because the resulting minimum cycle basis of \(\digraph\) contains a non-integral set of linearly independent cycles coming from
\(\odigraph\), it cannot be integral.

If \(\vertex\) is an isolated vertex of \(\digraph\) and
\(\digraph \setminus \vertex\) is \notmcbeq, then trivially \(\digraph\) is also
\notmcbeq.
Since vertex deletion can be realized as a sequence of deletions of incident
edges followed by the deletion of a disjoint vertex, the claim also holds for
vertex deletion.

Finally, consider edge contraction.
Let \(\Arco =(u,v) \in \Darcs\) and define
\(\odigraph \coloneqq \quotient{\digraph}{\Arco}\).
Assume that \(\odigraph\) is \notmcbeq.
Extend any weight assignment on \(\odigraph\) to \(\digraph\) by setting
\(\weight(\Arco) \coloneqq 0\).
Then any minimum cycle basis \(\cyclebasis\) of \(\digraph\) induces a cycle
basis
\[
\ocyclebasis \coloneqq
\bigl\{ \quotient{\cycle}{\Arco} \mid \cycle \in \cyclebasis \bigr\}
\]
of \(\odigraph\)  of equal weight. In general the quotient of a cycle \(\quotient{\cycle}{\Arco}\) is not a cycle if \(u,v \in \cycle\) but \((u,v)= \Arco \notin \Arcs{\cycle}\). However, since \(\weight(\Arco)=0\), any cycle in a minimum cycle basis containing both vertices \(u\) and \(v\) must also contain \(\Arco\). So \(\quotient{\cycle}{\Arco}\) is indeed a cycle for any \(\cycle\) in a minimum cycle basis of \(\digraph\).

Moreover, this operation preserves the determinant of the cycle basis.
Indeed, the set of chords of \(\odigraph\) corresponding to a spanning tree
\(\otree\) coincides with the set of chords of \(\digraph\) corresponding to the
spanning tree \(\tree \coloneqq \otree \cup \{\Arco\}\).
Thus there is a bijection between minimum cycle bases of \(\digraph\) and \(\odigraph\)
that preserves both weights and integrality.
Consequently, if \(\odigraph\) is \notmcbeq, then so is \(\digraph\).
\end{proof}

The Robertson--Seymour theorem~\cite{BIENSTOCK1995,ROBERTSON2004} states that every minor-closed family of graphs can be characterized by a finite set of forbidden minors. A direct consequence of Theorem~\ref{thm: optin is minor closed} is the following:

\begin{corollary}
    There exists a finite minimal set \(\forbiddenminors\) of \notmcbeq\ graphs that completely characterizes the class of \notmcbeq\ (and hence \mcbeq) graphs.
\end{corollary}

By Observation~\ref{obs:undirected-minors}, we can view \(\forbiddenminors\) as a set of undirected graphs. For any fixed graph \(H\), deciding whether an undirected graph \(\graph\) contains \(H\) as a minor can be done in polynomial time in the size of \(\graph\), with a constant factor that depends superpolynomially on the size of \(H\) \cite{ROBERTSON1995}. Consequently, deciding whether \(\graph\) contains a fixed member of \(\forbiddenminors\) as a minor can also be performed in polynomial time.

\begin{corollary}
\label{cor: complexity of opt-in decidability}
    There exists a polynomial-time algorithm to decide whether a given graph is \mcbeq\ or \notmcbeq.
\end{corollary}

However, this result is non-constructive: in order to explicitly obtain such an algorithm, one would first need to determine the set of forbidden minors \(\forbiddenminors\). 

\begin{observation}
    We could also have defined \mcbeq{} in terms of minimum undirected cycle bases. While this changes the set of \mcbeq{} graphs, the results and proofs of this section remain valid.
\end{observation}

\section{Basic results for graphs in \(\forbiddenminors\)}
\label{sec: results for X}

Identifying the minimal set \(\forbiddenminors\) of forbidden minors would allow one to efficiently determine whether a minimum cycle basis algorithm can be safely used to solve the minimum integral cycle basis problem on a given graph. More broadly, we believe that identifying such forbidden minors could contribute to understanding the complexity of the minimum integral cycle basis problem. In particular, graphs in \(\forbiddenminors\) exhibit the essential obstructions to minimum cycle basis algorithms and may therefore provide useful gadgets for complexity reductions, or provide better understanding of the structure of the problem.

Proving that a given graph \(\digraph\) belongs to \(\forbiddenminors\) requires establishing two conditions. First, one must prove that \(\digraph\) is \notmcbeq, which can be accomplished by exhibiting a single edge-weight assignment for which no minimum cycle basis is integral. Second, one must prove that every graph obtained from \(\digraph\) by edge deletion or contraction is \mcbeq. Although Proposition \ref{prop: MICBP in optin is in P} guarantees the existence of a polynomial-time algorithm for deciding \mcbeq-ness in general, this guarantee is non-constructive and itself presupposes knowledge of \(\forbiddenminors\). 

In this section, we establish several structural and combinatorial properties of graphs in \(\forbiddenminors\).

\begin{lemma}\label{lem: forbidden minors are 2-connected min degree 3}
All graphs in \(\forbiddenminors\) are $2$-connected with minimum degree $3$.
\end{lemma}

\begin{proof}
For \(2\)-connectivity, if \(\digraph \in \forbiddenminors\) has a cut vertex, then it decomposes into blocks with \(\Dcyclespace = \bigoplus_{\block\in\blocks}\cyclespace{\block}{\bR}\) (Observation~\ref{obs:blocks}), so that a minimum cycle basis of \(\graph\) restricts to a minimum cycle basis of each block independently, and \(\graph\) is \notmcbeq\ iff some block is. Therefore, if \(\digraph \in \forbiddenminors\), then the number of blocks is exactly \(1\) and \(\digraph\) is \(2\)-connected.
 
Concerning the minimum degree, a vertex of degree \(\le1\) lies on no cycle and can be deleted without changing the cycle space. 
A vertex \(v\) of degree exactly \(2\) subdivides the path between its two neighbors; contracting one of its two incident arcs leaves the cycle space structurally unchanged, so that \(\digraph\) and \(\digraph/e\) again share the same \mcbeq-ness status, precluding minimality.
\end{proof}

\begin{definition}
For each cycle basis $\cyclebasis$ of $\digraph$, we define its \emph{coefficient vector} $c_\cyclebasis \in \mathbb Z_{\geq 0}^{\Darcs}$ by
\[
c_{\cyclebasis}(\Arc) \coloneqq |\{C \in \cyclebasis : \Arc \in C\}|.
\]
\end{definition}

With this notation, we have $w(\cyclebasis) = c_\cyclebasis^\top w$.

Let \(\optbases\) denote the set of directed cycle bases that are the unique minimum-weight basis for some weight \(\weight\ge\mathbf 0\), and split it as \(\optbases = \optintbases \sqcup \optnotintbases\), where \(\optintbases\) (respectively \(\optnotintbases\)) consists of the integral (respectively non-integral) members of \(\optbases\).

\begin{observation}\label{obs:optout-basic}
Recall that by Proposition~\ref{prop: MICBP in optin is in P} and the perturbation argument in its proof, a graph is \notmcbeq{} if and only if there exists a weight $w$ such that the minimum directed cycle basis $\cyclebasis$ w.r.t.\ $w$ is unique and non-integral. Therefore, a graph $\digraph$ is \notmcbeq\ if and only if $\optnotintbases \neq \emptyset$; and for each $\cyclebasis \in \optnotintbases$, there exists a weight vector $w \geq \mathbf 0$ such that
\[
c_\cyclebasis^\top w < c_{\cyclebasis'}^\top w \quad \text{for all } \cyclebasis' \in \mathfrak{B}_I.
\]
\end{observation}

\begin{lemma}\label{lem: minimum not integer cycle basis in X are greater than two}
If \(\digraph \in \forbiddenminors\) and \(\cyclebasis \in \optnotintbases\), then \(c_{\cyclebasis} \geq \mathbf{2}\).
\end{lemma}
\begin{proof}
If not, then some arc $e$ has $c_{\cyclebasis,e}=1$ (recall $c_\cyclebasis\ge\mathbf 1$ since \(\digraph\) is $2$-connected, Lemma~\ref{lem: forbidden minors are 2-connected min degree 3}); let $C$ be the unique cycle of $\cyclebasis$ containing $e$.
Since $\digraph\setminus\{e\}$ is \mcbeq\ by minimality of \(\digraph \in \forbiddenminors\), it has an integral cycle basis $\cyclebasis_e$ with weight at most that of $\cyclebasis\setminus\{C\}$, so $\cyclebasis' \coloneqq \cyclebasis_e\cup\{C\}$ is an integral basis of $\digraph$ with $\weight(\cyclebasis')\leq \weight(\cyclebasis)$. By the uniqueness assumption, $\cyclebasis = \cyclebasis'$, contradicting non-integrality of \(\cyclebasis\).
\end{proof}

Interestingly, a graph being in \(\forbiddenminors\) forces some balancing constraints on the weight \(\weight\) that induces a non-integral unique minimum cycle basis. In particular, \(\weight\) must be a strictly metric weight, that is, \(\weight(\Arc) < d_{\digraph\setminus\Arc}(u,v)\) for every arc
\(\Arc=(u,v)\in\Darcs\), where \(d_{\digraph\setminus\Arc}\) is
shortest-path distance in \(\digraph\setminus\Arc\).

\begin{lemma}\label{lem: X metric weight}
    Let \(\digraph \in \forbiddenminors\) and let $w \geq \mathbf {0}$ be a weight vector such that $\cyclebasis \in \optnotintbases$ is the unique minimum directed cycle basis for $w$. Then $w$ is strictly metric.
\end{lemma}
\begin{proof}
    Since every graph admits a minimum cycle basis that is isometric~\cite{cyclebasisintro}, uniqueness of \(\cyclebasis\) implies it must be isometric: for any cycle \(\cycle \in \cyclebasis\) and any two vertices \(u,v \in \vertices{\cycle}\), one of the two paths connecting \(u\) and \(v\) along \(\cycle\) is a shortest path in \(\graph\).

    We proceed by contradiction. If \(\weight\) is not metric, there exists an arc \(\Arc = (u,v) \in \Darcs\) whose weight is larger than the weight of a shortest path between \(u\) and \(v\) in \(\graph \setminus \Arc\).  By the perturbation argument in Proposition~\ref{prop: MICBP in optin is in P}, we can assume that a set of arcs is uniquely identified by its weight, so that the shortest path \(\walk_{u,v}\) is unique. Since the cycles in \(\cyclebasis\) are isometric, any \( \cycle\in\cyclebasis\) containing \(\Arc\) must be exactly \(\cycle=\Arc\walk_{u,v}\). Therefore, \(\Arc\) is contained in exactly one cycle in \(\cyclebasis\). This contradicts $c_\cyclebasis(e) \geq 2$ (Lemma~\ref{lem: minimum not integer cycle basis in X are greater than two}).
\end{proof}

Let \(\cycle\) be a cycle of \(\digraph\). An arc \(\Arc=(u,v)\in
\Darcs\setminus\Arcs\cycle\) with \(u,v\in\vertices\cycle\) non-consecutive
on \(\cycle\) is a \emph{chord} of \(\cycle\). A cycle with no chord is
\emph{chordless} (equivalently, \emph{non-chordal}, or \emph{induced}).

\begin{lemma}[Chordless minimum bases under strict metricity]\label{lem:chordless-strict-metric}
If \(\weight\) is strictly metric on \(\graph\), every minimum-weight directed
cycle basis of \(\digraph\) consists entirely of chordless cycles.
\end{lemma}
\begin{proof}
Let \(\cyclebasis\) be a minimum-weight cycle basis and suppose some
\(\cycle\in\cyclebasis\) has a chord \(\Arc=(u,v)\), splitting \(\cycle\)
into its two subpaths \(\walk_{u,v},\walk_{v,u}\) (so
\(\cycle=\walk_{u,v}\cup\walk_{v,u}\)). Let \(\cycle_1\coloneqq
\walk_{u,v}\cup\{\Arc\}\) and \(\cycle_2\coloneqq\walk_{v,u}\cup\{\Arc\}\)
be the two cycles obtained by closing each subpath with the chord, so
that \(\cycle = \cycle_1 + \cycle_2\). By Horton's exchange theorem \cite{horton1987}, there exist \(\cycle'\in\{\cycle_1,\cycle_2\}\) such that \(\cyclebasis\setminus\{\cycle\}\cup\{\cycle'\}\) is a cycle basis. Say \(\cycle' = \cycle_1\), since \(\walk_{v,u}\) is a path from \(u\) to \(v\) avoiding \(\Arc\),
strict metricity gives \(\weight(\Arc) < d_{\digraph\setminus\{\Arc\}}(u,v)
\le \weight(\walk_{v,u})\), hence
\[
\weight(\cycle_1) = \weight(\walk_{u,v})+\weight(\Arc)
< \weight(\walk_{u,v})+\weight(\walk_{v,u}) = \weight(\cycle).
\]
So \(\weight(\cyclebasis') < \weight(\cyclebasis)\), contradicting
minimality of \(\cyclebasis\). Hence \(\cyclebasis\) has no cycle with a chord.
\end{proof}

\begin{lemma}\label{lem:simple-graph}
    Any $\digraph \in \forbiddenminors$ is a simple graph.
\end{lemma}
\begin{proof}
    Let $\cyclebasis \in \optnotintbases$ be the unique non-integral minimum cycle basis w.r.t.\ the strictly metric weights $w$ (Lemma~\ref{lem: X metric weight}). Suppose that $\digraph$ has multiple arcs $e_1, \dots, e_k$ between the vertices $u$ and $v$, ordered by ascending weight. As $\cyclebasis$ must be isometric and consist of simple cycles \cite{cyclebasisintro}, if a cycle $\cycle \in \cyclebasis$ contains $e_i$ for some $i \geq 2$, its only arcs are $e_i$ and $e_1$. We conclude that in $\cyclebasis$, there is a unique cycle containing $e_i$, contradicting Lemma~\ref{lem: minimum not integer cycle basis in X are greater than two}.
    
    The same argument that shows 2-connectivity (see Lemma~\ref{lem: forbidden minors are 2-connected min degree 3}) also shows that $\digraph$ cannot have self-loops.
\end{proof}

We summarize the structural findings of this section as follows:

\begin{proposition}\label{prop:structure-summary}
    Let $\digraph \in \forbiddenminors$. Then $\digraph$ is a 2 -connected simple graph of minimum degree $3$. Moreover, there exists a strictly metric weight $\weight \geq \mathbf 0$ such that the corresponding minimum directed cycle basis $\cyclebasis$ is unique, non-integral, and contains only chordless cycles.
\end{proposition}
\begin{proof}
    This is a direct consequence of Proposition~\ref{prop: MICBP in optin is in P}, Lemma~\ref{lem: X metric weight},  Lemma~\ref{lem:chordless-strict-metric}, and Lemma~\ref{lem:simple-graph}.
\end{proof}

\section{Determining membership in \(\forbiddenminors\)}\label{sec: forbidden minors alg}
\newcommand{\convex}{\operatorname{conv}}

In this section, we develop a recursive algorithm that can decide whether a graph is opt-in or opt-out based on determining membership in $\forbiddenminors$.

\subsection{Certificates for non-membership in \(\forbiddenminors\)}

We can use Lemma~\ref{lem: minimum not integer cycle basis in X are greater than two} to obtain a certificate that a graph is not a forbidden minor. We define the cone
\[ \intcone \coloneqq \left\{\sum_{i=1}^k \lambda_i c_{\cyclebasis_i} : k \in \mathbb N, \cyclebasis_1, \dots, \cyclebasis_k \in \optintbases \right\} + \bR_{\geq0}^{\Darcs}\]

and observe the following:
\begin{lemma}\label{lem:dominate-2}
    If $\mathbf{2} \in \intcone$, then $\digraph \notin \forbiddenminors$.
\end{lemma}
\begin{proof}
    In this case, there are integral cycle bases \(\cyclebasis_1,\dots,\cyclebasis_k\) such that \[
\mathbf{2} \geq \sum_{i=1}^k \lambda_i c_{\cyclebasis_i}, \]
so that for any non-negative \(\weight\) holds
\[
\weight^\top\mathbf{2}
\geq
\sum_{i=1}^k\lambda_i\weight(\cyclebasis_i)
\geq
\weight(\cyclebasis_j)
\]
for some \(j\in\{1,\dots,k\}\). If $\digraph \in \forbiddenminors$ as certified by $\cyclebasis \in \optnotintbases$, then $w(\cyclebasis) = w^\top c_\cyclebasis \geq w^\top \mathbf{2}$ by Lemma~\ref{lem: minimum not integer cycle basis in X are greater than two}, so that $w(\cyclebasis) \geq w(\cyclebasis_j)$, contradiction.
\end{proof}

We can now see that while all planar graphs are opt-in, the family of opt-in graphs is strictly larger than the family of planar graphs, as demonstrated by the complete graph on five vertices:

\begin{example}[$K_5$ is opt-in]\label{ex: K5 is optin}
\vspace{1em}
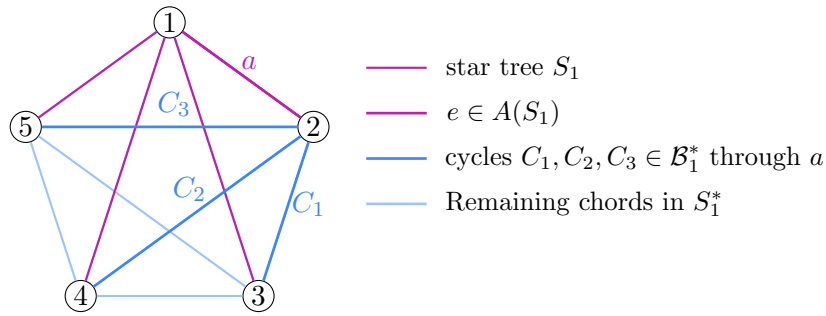
\begin{figure}[h!]
    \centering
    \begin{tikzpicture}[scale = 2,
    every node/.style={circle,draw,minimum size=3mm,inner sep=1pt},
    edge/.style={thick},
    treeedge/.style={edge,draw=\treecol},
    otheredge/.style={edge,draw=skyblue!50}]

\coordinate (v1) at (90:1);
\coordinate (v2) at (18:1);
\coordinate (v3) at (-54:1);
\coordinate (v4) at (-126:1);
\coordinate (v5) at (162:1);

\node (1) at (v1) {$1$};
\node (2) at (v2) {$2$};
\node (3) at (v3) {$3$};
\node (4) at (v4) {$4$};
\node (5) at (v5) {$5$};

\foreach \i/\j in {1/2,1/3,1/4,1/5,2/3,2/4,2/5,3/4,3/5,4/5} {
    \draw[otheredge] (\i) -- (\j);
}

\foreach \j in {2,3,4,5} {
    \draw[treeedge] (1) -- (\j);
}

\draw[\ared,line width=1pt] (1) -- (2);
\node[draw=none,text=\ared,inner sep=1pt] at ($(1)!0.5!(2)+(0.05,0.08)$) {$a$};

\draw[\bluecycle,line width=1pt] (2) -- (3);
\draw[\bluecycle,line width=1pt] (2) -- (4);
\draw[\bluecycle,line width=1pt] (2) -- (5);

\node[draw=none,text=\bluecycle,inner sep=1pt] at ($(2)!0.5!(3)+(0.15,0.05)$) {$C_1$};
\node[draw=none,text=\bluecycle,inner sep=1pt] at ($(2)!0.5!(4)+(-0.05,+0.15)$) {$C_2$};
\node[draw=none,text=\bluecycle,inner sep=1pt] at ($(2)!0.5!(5)+(0.03,+0.15)$) {$C_3$};

\begin{scope}[shift={(1.3,0.7)}, every node/.style={draw=none,circle=false,inner sep=1pt,minimum size=0pt,anchor=west,font=\small}]
    \draw[treeedge] (0,0) -- (0.4,0);
    \node at (0.5,0) {star tree $\startree_1$};

    \draw[vividorchid, line width = 1pt] (0,-0.3) -- (0.4,-0.3);
    \node at (0.5,-0.3) {$\Arc \in \Arcs{\startree_1}$};

    \draw[skyblue,line width=1pt] (0,-0.6) -- (0.4,-0.6);
    \node at (0.5,-0.6)  {cycles $C_1,C_2,C_3 \in \cyclebasis^*_1$ through $a$};

    \draw[skyblue!50,line width=1pt] (0,-0.9) -- (0.4,-0.9);
    \node at (0.5,-0.9)  {Remaining chords in \(\startree_1^*\)};
\end{scope}

\end{tikzpicture}
\caption{Fundamental cycle basis induced by the star $\startree_1$ centered at vertex $1$ in $\completegraph{5}$. A fixed $\Arc \in \Arcs{\startree_1}$ is contained in exactly three fundamental cycles $C_1, C_2, C_3$.}\label{fig: K5 star base}
\end{figure}

For the complete graph $K_5$ we can consider the ``star'' spanning trees $\startree_\vertex$ at each vertex $\vertex \in \{1, \dots, 5\}$. For the corresponding cycle basis \(\cyclebasis^*_\vertex\), we have \(c_{\cyclebasis^*_\vertex}(\Arc)=3\) if \(\Arc \in \Arcs{\startree_\vertex}\), since \(\Arc\) is contained in exactly three cycles of \(\cyclebasis^*_\vertex\), and \(c_{\cyclebasis^*_\vertex}(\Arc)=1\) otherwise (see Figure~\ref{fig: K5 star base}). Averaging the coefficient vectors of the five star bases \(\cyclebasis^*_\vertex\), \(\vertex\in\{1,\ldots,5\}\), yields
\[
\mathbf{2}
= \frac{2}{9}\mathbf{9}
= \frac{2}{9}\sum_{\vertex=1}^5 c_{\cyclebasis^*_\vertex}.
\]
 By Lemma~\ref{lem:dominate-2}. $\completegraph 5 \notin \forbiddenminors$. Since for every $e \in A(K_5)$, $\completegraph 5\setminus\{e\}$ is planar and hence \mcbeq, it follows that $\completegraph 5$ itself is \mcbeq.
\end{example}

In general, even for \mcbeq\ graphs, \(\mathbf{2}\) need not be dominated by a convex combination of integral cycle basis vectors.

\begin{example}[\(\mathbf 2 \notin \intcone\), but \(\digraph \notin \forbiddenminors\)]\label{ex: not 2-cycle basis coverable}
 The graph $\digraph$ obtained from the complete graph $K_7$ on the vertices $1, \dots, 7$ with the two arcs $(5, 7)$ and $(6, 7)$ removed 
has $|\Darcs|=19$ and cyclomatic number $\mu(D)=13$. We will see later that $K_7$ and hence $\digraph$ are opt-in. Under the uniform
weight $\weight=\mathbf 1$, the minimum cycle basis has weight
$39=3\times13$, achieved by a basis of $13$ triangles. If $\mathbf 2 \in \intcone$, i.e.,
$\mathbf 2 \geq \sum_{i=1}^k\lambda_i c_{\cyclebasis_i}$, then $38=2\times19=
\weight^\top\mathbf 2\ge\min_{i=1,\dots,k}\weight(\cyclebasis_i)$, giving
a cycle basis of weight below $39$ -- impossible.
\end{example}

We therefore generalize the argument above.

\begin{lemma}\label{lem:domination-certificate}
    Let \(\Cervec\subseteq\mathbb Z_{\geq0}^{\Darcs}\) be a finite set such that \(\Cervec \subset \intcone\). If for every non-integral $\cyclebasis$ exists a vector $b \in \Cervec$ such that $c_\cyclebasis \geq b$, then $\digraph \notin \forbiddenminors$.
\end{lemma}
\begin{proof}
    This is analogous to Lemma~\ref{lem:dominate-2}, replacing $\mathbf{2}$ with $b$.
\end{proof}

Observation~\ref{obs:optout-basic} implies that a graph is \mcbeq\ if and only if every strictly supported Pareto optimal solution with respect to the cost vectors \(c_{\cyclebasis}\) is integral. While this connection implies that the \mcbeq-decision problem could theoretically be modeled and solved using traditional multi-objective optimization frameworks (such as Benson's algorithm~\cite{Benson1998}), such an approach is highly impractical: the computational complexity of multi-objective scalarization scales poorly with the number of objectives, which in our setting corresponds to the cardinality of the arc set \(|\Darcs|\). We therefore develop an ad-hoc recursive algorithm which classifies whether a graph \(\digraph\) is \mcbeq, and if not, determines the forbidden minors of \(\digraph\). This is done by recursively checking whether each candidate is \notmcbeq\ or not in \(\forbiddenminors\): non-membership in \(\forbiddenminors\) is certified by finding a finite set \(\Cervec\) as in Lemma~\ref{lem:domination-certificate}. Otherwise, if during the search for \(\Cervec\) a non-integral cycle basis is encountered, then the algorithm outputs \notmcbeq (cf.\ Proposition~\ref{prop: MICBP in optin is in P}). The same procedure is then applied recursively to all $1$-minors (minors obtained by deleting or contracting one arc), until one of the known \mcbeq\ cases is met, e.g., \(\card{\vertices{\digraph}}\leq 5\) (Example~\ref{ex: K5 is optin}). If all $1$-minors are \mcbeq\ and \(\digraph\) is not in \(\forbiddenminors\), then \(\digraph\) is \mcbeq\ as well; otherwise, if \(\digraph\) is \notmcbeq\ and every proper minor is \mcbeq\, this proves \(\digraph \in \forbiddenminors\). 

 We then apply this algorithm to characterize \mcbeq-ness for complete graphs, showing that \(\completegraph{n}\) is \mcbeq\ if and only if \(n\leq 7\). We describe the details of the algorithm in the subsequent subsections.

\subsection{The domination linear program}

To check that a vector \(\cervec\in\Cervec\) is dominated by a convex combination of integral cycle bases, i.e., $\cervec \in \intcone$, our algorithm relies on the following linear program, which has optimal value \(0\) if and only if \(\cervec\) is dominated by a convex combination of the coefficient vectors \(c_{\cyclebasis}\) for \(\cyclebasis \in \mathfrak{B}_O\):
\begin{equation}\tag{LP}\label{LP}
\begin{aligned}
    &\text{minimize} & z & \\
    &\text{s.t.} & \sum_{\cyclebasis\in\mathfrak{B}_O}\lambda_\cyclebasis c_\cyclebasis - z\mathbf{1} &\le \cervec, \\
    & & \sum_{\cyclebasis\in\mathfrak{B}_O}\lambda_\cyclebasis &= 1, \\
    & & \lambda_\cyclebasis &\geq0, \ \cyclebasis\in\mathfrak{B}_O, \qquad z\ge0.
\end{aligned}
\end{equation}
The corresponding dual problem is:
\begin{align*}
    &\text{maximize} & y-\cervec^\top w\\
    &\text{s.t.} & w^\top c_\cyclebasis &\geq y,\ \cyclebasis\in\mathfrak{B}_O, \qquad \mathbf{1}^\top w\le 1,\ w\ge\mathbf{0}.
\end{align*}
Executing column generation given a restricted dual solution $(w^*,y^*)$ amounts to finding a cycle basis $\cyclebasis$ that minimizes $(w^*)^\top c_\cyclebasis$, via a standard minimum cycle basis algorithm (e.g., De Pina's algorithm~\cite{dePina1995,cyclebasisintro}). Due to the equivalence of separation and optimization, the LP can hence be solved in polynomial time. If the pricing step ever produces a non-integral cycle basis, \(\digraph\) is \notmcbeq{} by Proposition~\ref{prop: MICBP in optin is in P}. Otherwise, upon termination, the LP determines whether \(\cervec\) is dominated. If so ($z^* = 0)$, then all $\cyclebasis \in \optbases$ with $\lambda_\cyclebasis > 0$ are also in $\optintbases$, so that $\cervec \in \intcone$. If \(\cervec\) is not dominated ($z^* > 0$), it cannot serve as a certificate on its own and must be systematically augmented into tighter lower-bounding vectors, as described in the next section.

\subsection{A recursive parity-based domination test}

The last step consists of finding a set \(\Cervec\subset\bZ_{\ge0}^{\Darcs}\) that dominates at least one element \(\mathfrak{B}_N\) and is itself dominated by \(\convex(\{c_{\cyclebasis'}:\cyclebasis'\in\optintbases\})\). For a simple cycle $\cycle$ of $\digraph$, $e_\Arc$ denotes the standard basis vector at $\Arc \in \Darcs$, and $\mathbf 1_\cycle$ the incidence vector of $\cycle$'s arc set. Since the algorithm tests membership in \(\forbiddenminors\) at each step, we assume \(\digraph\in\forbiddenminors\) throughout and derive an algorithmically checkable contradiction otherwise. By Lemma~\ref{lem: minimum not integer cycle basis in X are greater than two}, every \(\cyclebasis\in\optnotintbases\) has \(c_\cyclebasis\ge\mathbf 2\) whenever \(\digraph\in\forbiddenminors\). If \(\mathbf 2 \in \intcone\), we are done with \(\Cervec=\{\mathbf{2}\}\); otherwise we must consider vectors larger than \(\mathbf 2\) that still provably lower bound \(c_\cyclebasis\). The following lemma describes the augmentation step.

\begin{lemma}[Recursive parity dichotomy]\label{lem:parity-dichotomy-recursive}
Let $\cyclebasis$ be a cycle basis of $\digraph$ with
$c_{\cyclebasis}\ge \domvec \in \bZ_{\geq0}^{\Darcs}$, where \(\domvec(\delta(\vertex)) \in \bZ^{\Darcs}\) is even for every \(\vertex\in\Dvertices\) and $c_{\cyclebasis}\ne \domvec$. Then at least
one of the following holds:
\begin{enumerate}
    \item\label{item:edge-case} there is an arc $\Arc\in\Darcs$ with $c_{\cyclebasis}\ge \domvec' \coloneqq\domvec+2e_{\Arc}$;
    \item\label{item:cycle-case} there is an undirected cycle $\cycle$ of
    $\digraph$ with $c_{\cyclebasis}\ge \domvec'\coloneqq\domvec+\mathbf{1}_C$.
\end{enumerate}
Furthermore \(\domvec'(\delta(\vertex)) \in \bZ^{\Darcs}\) is even for every \(\vertex\in\Dvertices\).
\end{lemma}

\begin{proof}
If \ref{item:edge-case} does not hold then \(\domvec+\mathbf{1} \geq c_{\cyclebasis}\geq \domvec\). The arc set
\begin{align*}
F &= \{\Arc \in \Darcs\mid c_\cyclebasis(\Arc) - \domvec(\Arc) \text{ is odd}\} \\
&= \{\Arc \in \Darcs\mid c_\cyclebasis(\Arc) - \domvec(\Arc) = 1\},
\end{align*}
is the support of \((\domvec + c_\cyclebasis) \bmod 2\). Since both \(\domvec\) and \(c_\cyclebasis\) have even degree at every \(\vertex\), $F$ must hence contain a simple cycle \(\cycle\), and \(c_\cyclebasis(\Arc) \geq \domvec(\Arc) + 1\) for every arc \(\Arc \in \cycle\).
\end{proof}

We can now formulate an algorithm (Algorithm~\ref{alg:refined-domination}) that can decide whether a graph is \notmcbeq{} or not a forbidden minor by implicitly enumerating a set $\Cervec$ satisfying the hypotheses of Lemma~\ref{lem:domination-certificate} making use of the restrictions imposed by Lemma~\ref{lem:parity-dichotomy-recursive}.

\begin{algorithm}
\caption{Refined domination test}
\label{alg:refined-domination}
\begin{algorithmic}[1]
\Require A $2$-connected simple graph $\digraph$.
\Ensure $\digraph$ is \notmcbeq{} or not in $\forbiddenminors$
\Procedure{DominationTest($\digraph$)}{}
\State $\mathcal O \gets \{\mathbf{2}\}$ \Comment{open nodes}
\State $\mathcal Q \gets \varnothing$ \Comment{closed, dominated nodes}
\While{$\mathcal O \neq \varnothing$}
    \State Extract some $\domvec \in \mathcal O$
    \State $z^* \gets $ value of \eqref{LP} with target $\domvec$
    \If{$z^* = 0$ and the pricing step \eqref{LP} produced only integral cycle bases}
        \State $\mathcal Q \gets \mathcal Q \cup \{\domvec\}$
    \Else
        \If{the pricing step in \eqref{LP} produced a non-integral cycle basis}
            \State \Return ``$\digraph$ is \notmcbeq{}''
        \EndIf
        \ForAll{$\Arc \in \Darcs$}
            \State $\domvec' \gets \domvec + 2e_\Arc$
            \If{$\domvec' \not\ge b$ for all $b\in\mathcal Q$ \textbf{and} 
                  $\domvec'(\Arc') \le \cyclerank\digraph$ for all $\Arc'\in\Darcs$}
                \State $\mathcal O \gets \mathcal O \cup \{\domvec'\}$
            \EndIf
        \EndFor
        \ForAll{simple undirected cycles $\cycle$ of $\digraph$}
            \State $\domvec' \gets \domvec + \mathbf 1_\cycle$
            \If{$\domvec' \not\ge b$ for all $b\in\mathcal Q$ \textbf{and}
                  $\domvec'(\Arc) \le \cyclerank\digraph$ for all $\Arc\in\Darcs$}
                \State $\mathcal O \gets \mathcal O \cup \{\domvec'\}$
            \EndIf
        \EndFor
    \EndIf
\EndWhile
\State \Return ``$\digraph\notin\forbiddenminors$''
\EndProcedure
\end{algorithmic}
\end{algorithm}

\begin{theorem}[Refined domination test]\label{thm:refined-domination}
Algorithm~\ref{alg:refined-domination} correctly decides whether $\digraph$ is \notmcbeq{} or not in $\forbiddenminors$.
\end{theorem}

\begin{proof}
Since for all cycle bases $\cyclebasis$ holds \(c_\cyclebasis \leq \cyclerank\digraph \mathbf{1} \), we can restrict the search to only \(\domvec \leq \cyclerank\digraph \mathbf{1} \). The procedure hence terminates in finitely many steps, since the number of integral vectors \(\mathbf 2 \leq \domvec \leq \cyclerank\digraph \mathbf{1}\) is finite. 

Suppose that $\digraph \in \forbiddenminors$. Then there is a cycle basis $\cyclebasis \in \optnotintbases$ with $c_\cyclebasis \geq \mathbf{2}$ (Lemma~\ref{lem: minimum not integer cycle basis in X are greater than two}). The domination linear program \eqref{LP} then must detect \notmcbeq{}-ness for $b = c_\cyclebasis$, as the optimal value $z^* = 0$ can only be reached by generating the column for $\cyclebasis$. As ``$\digraph \notin \forbiddenminors$'' can only be returned after $c_\cyclebasis$ has been enumerated by the algorithm, we therefore decide correctly. 

Assume now that $\digraph \notin \forbiddenminors$. There is nothing to show if $\digraph$ is also \notmcbeq{}. If $\digraph$ is \mcbeq{}, then by Proposition~\ref{prop: MICBP in optin is in P}, the pricing step of \eqref{LP} will never require a non-integral cycle basis, so that the algorithm returns that $\digraph \notin \forbiddenminors$.
\end{proof}

\subsubsection*{Pruning techniques}

Since constructing \(\Cervec\) via Lemma~\ref{lem:parity-dichotomy-recursive} alone yields enormous \(\opennodes\), to make algorithm \ref{alg:X search} usable we use various pruning techniques to diminish the size of the set \(\opennodes\).

\paragraph*{Chordless cycles.\;} As we know by Lemma~\ref{lem:chordless-strict-metric} that we can restrict to chordless cycle bases, we can restrict the number of candidates $b$ in Algorithm~\ref{alg:refined-domination} by requiring
\[ b(e) \leq \min \left\{\card{\{\cycle \text{ a chordless cycle of }\digraph : \Arc\in\cycle\}}, \cyclerank\digraph \right\} \]
for all $e \in \Darcs$.

\paragraph*{Arc count pruning.\;} Depending on the structure of the graph, we may know without running \eqref{LP} that \(\domvec\) cannot be in $\intcone$, see also Example~\ref{ex: K7 is optin} later on. For example, this occurs if $\mathbf 1^\top\domvec$ is less than the weight of a minimum cycle basis of $\graph$ under $\weight=\mathbf 1$.

\paragraph*{Orbit pruning.\;} A graph automorphism $\sigma\in\mathrm{Aut}(\graph)$ acts on $\Darcs$ such that for $b \in\bZ_{\ge0}^{\Darcs}$, we have $b \in \intcone$ if and only if $\sigma\cdot b \in \intcone$. In particular, it suffices to consider only one representative $b$ per orbit.

\subsection{An \mcbeq-\notmcbeq\ decision algorithm}
By executing a recursive search, we systematically determine the membership of $\digraph$ and its minors within the forbidden minor set $\forbiddenminors$. Algorithm \ref{alg:X search} establishes graph classification via three structural cases:
\begin{enumerate}
    \item If all proper minors of $\digraph$ are \mcbeq\ but $\digraph$ itself is \notmcbeq, then $\digraph$ is in $\forbiddenminors$.
    \item If all proper minors of $\digraph$ are \mcbeq\ and the refined domination test certifies that \(\digraph\) is not in \(\forbiddenminors\), then $\digraph$ is \mcbeq.
    \item If at least one proper minor of $\digraph$ is \notmcbeq, then $\digraph$ inherits this non-integrality by minor-closure (Theorem~\ref{thm: optin is minor closed}), meaning it is \notmcbeq\ but does not belong to the minimal set $\forbiddenminors$.
\end{enumerate}

\begin{algorithm}
\caption{Recursive Search for Forbidden Minors $\forbiddenminors$}
\label{alg:X search}
\begin{algorithmic}[1]
\Require A $2$-connected graph $\digraph$.
\Ensure The classification of $\digraph$ (\mcbeq\ or \notmcbeq) and the accumulated set of forbidden minors $\forbiddenminors_{\text{local}}$.
\Procedure{SearchMinors($\digraph$)}{}
    \If{$|V(\digraph)| \leq 5$}
        \Comment{Base case for small graphs: \(\completegraph 5\) is \mcbeq.}
        \State \Return $(\text{\mcbeq},\emptyset)$
    \EndIf
    \State $\text{is\_minimal} \gets \text{true}$
    \State $\text{graph\_status} \gets \text{\mcbeq}$
    \State $\forbiddenminors_{\text{local}} \gets \emptyset$
    \For{each isomorphism class of $2$-connected simple minors with minimum degree $3$ of $\digraph$ obtained by a single arc deletion or contraction, with representative $\digraph'$}
        \State $(\text{status}(\digraph'),\forbiddenminors') \gets
        \textsc{SearchMinors}(\digraph')$
        \State $\forbiddenminors_{\text{local}} \gets
        \forbiddenminors_{\text{local}} \cup \forbiddenminors'$
        \If{$\text{status}(\digraph') = \text{\notmcbeq}$}
            \State $\text{graph\_status} \gets \text{\notmcbeq}$
            \State \label{alg-line:minimality-flag}
            $\text{is\_minimal} \gets \text{false}$
        \EndIf
    \EndFor
    \If{$\text{is\_minimal} = \text{true}$}\label{alg-line:domination-test}
        \If{$\textsc{DominationTest}(\digraph) = \text{\notmcbeq}$}
            \State $\text{graph\_status} \gets \text{\notmcbeq}$
            \State $\forbiddenminors_{\text{local}} \gets
            \forbiddenminors_{\text{local}} \cup \{\digraph\}$
        \EndIf
    \EndIf
    \State \Return $(\text{graph\_status},\forbiddenminors_{\text{local}})$
\EndProcedure
\end{algorithmic}
\end{algorithm}

\begin{theorem}\label{thm:algo}
Algorithm~\ref{alg:X search} is correct.
\end{theorem}
\begin{proof}
We proceed by induction on the minor poset of $\digraph$. If some proper minor $\digraph'$ is verified \notmcbeq, then $\digraph$ is immediately \notmcbeq\ by minor-closure (Theorem~\ref{thm: optin is minor closed}). If $|V(\digraph)| \le 5$, then $\digraph$ is \mcbeq{} due to minor closure and Example~\ref{ex: K5 is optin}.

Otherwise, suppose all single-edge deletions and contractions of $\digraph$ are certified \mcbeq; by transitivity, every proper minor of $\digraph$ is \mcbeq, and Theorem~\ref{thm:refined-domination} applies. If the refined domination test terminates without ever finding a non-integral basis, $\digraph\notin\forbiddenminors$ by that proposition; combined with every proper minor being \mcbeq, $\digraph$ must itself be \mcbeq. Conversely, if the test finds a non-integral cycle basis, $\digraph$ is \notmcbeq; combined with every proper minor being \mcbeq, $\digraph$ is minimally \notmcbeq, i.e., $\digraph\in\forbiddenminors$.
\end{proof}

\subsection{Classification of complete graphs}

As an application of Algorithms~\ref{alg:refined-domination}~and~\ref{alg:X search}, we settle the \mcbeq{}/\notmcbeq{} status for complete graphs.

\begin{example}[$K_6$ is \mcbeq{}]\label{ex: K6 is optin}
We follow the execution of Algorithm~\ref{alg:X search} for \(\completegraph 6\). All proper minors are \mcbeq{}, so that we reach line~\ref{alg-line:domination-test} with $\forbiddenminors_{\text{local}} = \emptyset$ and $\texttt{is\_minimal} = \texttt{true}$. The refined domination test (Algorithm~\ref{alg:refined-domination}) for $K_6$ finds on the first iteration that $\mathbf{2} \in \intcone$. More precisely, for each of the six star bases $\cyclebasis_i^*$, each tree arc lies in $4$ of the $10$ cycles in $\cyclebasis_i^*$ and each of the $\binom 62$ arcs is a tree arc in exactly two of the six stars; summing over the six stars, $\tfrac16\sum_{i=1}^{6} c_{\cyclebasis_i^*} \equiv \mathbf{2}$. Therefore, $\mathbf{2}$ is a convex combination of the fundamental (and hence integral) cycle bases $\cyclebasis_i^*$.
So $\mathbf 2 \in \intcone$, and by Lemma~\ref{lem:dominate-2}, $\completegraph 6\notin\forbiddenminors$. Algorithm~\ref{alg:refined-domination} detects this by adding $\mathbf{2}$ to the dominated set $\mathcal Q$ and will terminate with $\mathcal O = \emptyset$ and ``$K_6 \notin \forbiddenminors$'' after the first iteration.
\end{example}

\begin{example}[$K_7$ is \mcbeq{}]\label{ex: K7 is optin}
Algorithm~\ref{alg:X search} proves \mcbeq{}-ness of \(\completegraph{7}\). We briefly sketch why $K_7 \notin \forbiddenminors$. For \(n \geq 7\), in \(\completegraph{n}\), $\mathbf 2$ is not dominated by a convex combination of coefficient vectors of integer cycle bases. For $K_7$, \eqref{LP} for $b = \mathbf{2}$ has the optimal value $z^* = \frac{1}{7}$. However, we know by Proposition~\ref{prop:structure-summary}, that if $K_7 \in \forbiddenminors$, then there is a cycle basis $\cyclebasis \in \optnotintbases$ consisting only of triangles. We then have $\mathbf{1}^\top c_\cyclebasis = 3 \cyclerank{K_7} = 45.$ Since $K_7$ has 21 arcs, Lemma~\ref{lem:parity-dichotomy-recursive} implies $c_\cyclebasis = \mathbf{2}+ \mathbf{1}_C$ for a triangle $C$. By orbit pruning, we only need to consider $b = \mathbf{2} + \mathbf{1}_C$ for a single triangle $C$. We find the optimal value $z^* = 0$ for \eqref{LP}: The vector $b$ is a convex combination of $21$ coefficient vectors of integral cycle bases. We conclude by Lemma~\ref{lem:domination-certificate} that $K_7 \notin \forbiddenminors$.
\end{example}

\begin{example}\label{ex: K8 is opt-out}

Figure~\ref{fig: smallest graph in X} shows a graph \(\digraph \in \forbiddenminors\) with $8$ vertices and $23$ arcs,
obtained by running Algorithm~\ref{alg:X search} on \(\completegraph{8}\).
Since \(\completegraph{7}\) is \mcbeq, so is every graph on at most seven
vertices (each such graph is a minor of \(\completegraph{7}\), and \mcbeq{}-ness is preserved under taking minors), so \(\digraph\) has the fewest
possible vertices among members of \(\forbiddenminors\). The weights
shown on the arcs of \(\digraph\) in the figure induce a unique non-integral minimum-weight cycle basis  of total weight $2490$, while the minimum-weight integral cycle basis has a weight of $2491$.
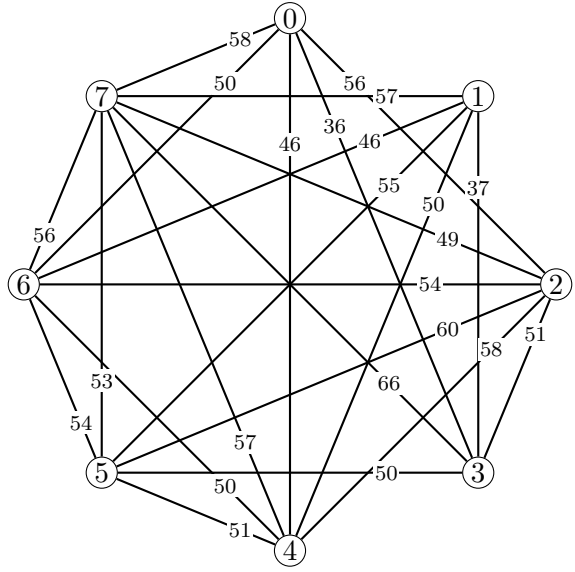
\begin{figure}[h!]

    \centering
\begin{tikzpicture}[scale = 2.2,
    every node/.style={circle,draw,minimum size=3mm,inner sep=1pt},
    edge/.style={thick},
    weightlabel/.style={font=\scriptsize, fill=white, inner sep=0.5pt, draw=none, shape=rectangle}]
    \foreach \i in {0,...,7} {
        \coordinate (v\i) at (90-\i*45:1.6);
        \node (\i) at (v\i) {$\i$};
    }
    \draw[edge] (0) -- node[weightlabel,pos=0.22] {56} (2);
    \draw[edge] (0) -- node[weightlabel,pos=0.22] {36} (3);
    \draw[edge] (0) -- node[weightlabel,pos=0.22] {46} (4);
    \draw[edge] (0) -- node[weightlabel,pos=0.22] {50} (6);
    \draw[edge] (0) -- node[weightlabel,pos=0.22] {58} (7);
    \draw[edge] (1) -- node[weightlabel,pos=0.22] {37} (3);
    \draw[edge] (1) -- node[weightlabel,pos=0.22] {50} (4);
    \draw[edge] (1) -- node[weightlabel,pos=0.22] {55} (5);
    \draw[edge] (1) -- node[weightlabel,pos=0.22] {46} (6);
    \draw[edge] (1) -- node[weightlabel,pos=0.22] {57} (7);
    \draw[edge] (2) -- node[weightlabel,pos=0.22] {51} (3);
    \draw[edge] (2) -- node[weightlabel,pos=0.22] {58} (4);
    \draw[edge] (2) -- node[weightlabel,pos=0.22] {60} (5);
    \draw[edge] (2) -- node[weightlabel,pos=0.22] {54} (6);
    \draw[edge] (2) -- node[weightlabel,pos=0.22] {49} (7);
    \draw[edge] (3) -- node[weightlabel,pos=0.22] {50} (5);
    \draw[edge] (3) -- node[weightlabel,pos=0.22] {66} (7);
    \draw[edge] (4) -- node[weightlabel,pos=0.22] {51} (5);
    \draw[edge] (4) -- node[weightlabel,pos=0.22] {50} (6);
    \draw[edge] (4) -- node[weightlabel,pos=0.22] {57} (7);
    \draw[edge] (5) -- node[weightlabel,pos=0.22] {54} (6);
    \draw[edge] (5) -- node[weightlabel,pos=0.22] {53} (7);
    \draw[edge] (6) -- node[weightlabel,pos=0.22] {56} (7);
\end{tikzpicture}

 \caption{Smallest graph in \(\forbiddenminors\) obtained with Alg.~\ref{alg:X search}, together with arc weights inducing a non-integral minimum cycle basis.}\label{fig: smallest graph in X}
\end{figure}
\end{example}

We can now obtain a complete characterization of \mcbeq-ness for complete graphs.

\begin{theorem}\label{thm:complete-characterization}
    The complete graph \(\completegraph{n}\) is \mcbeq\ if and only if \(n\leq 7\).
\end{theorem}

\begin{proof}
    By Example~\ref{ex: K6 is optin}, \(\completegraph{7}\) is \mcbeq. Since \mcbeq\ graphs are minor-closed, every complete graph \(\completegraph{n}\) with \(n\leq 7\) is \mcbeq.

    On the other hand, Example~\ref{ex: K8 is opt-out} shows that \(\completegraph{8}\) is \notmcbeq\ since it contains a forbidden minor.
\end{proof}

Beyond complete graphs, we have found the following results for generalized Petersen graphs.

\begin{example}
    Using Algorithm \ref{alg:X search}, starting from a search of the generalized Petersen graph \(P_{7,2}\), we found that the graphs in Figure~\ref{fig:forbiddeniminors} are in \(\forbiddenminors\), while \(P_{5,2}\) was found to be \mcbeq. 
    \begin{figure}[h]
        \centering
        \includegraphics[width=1\linewidth]{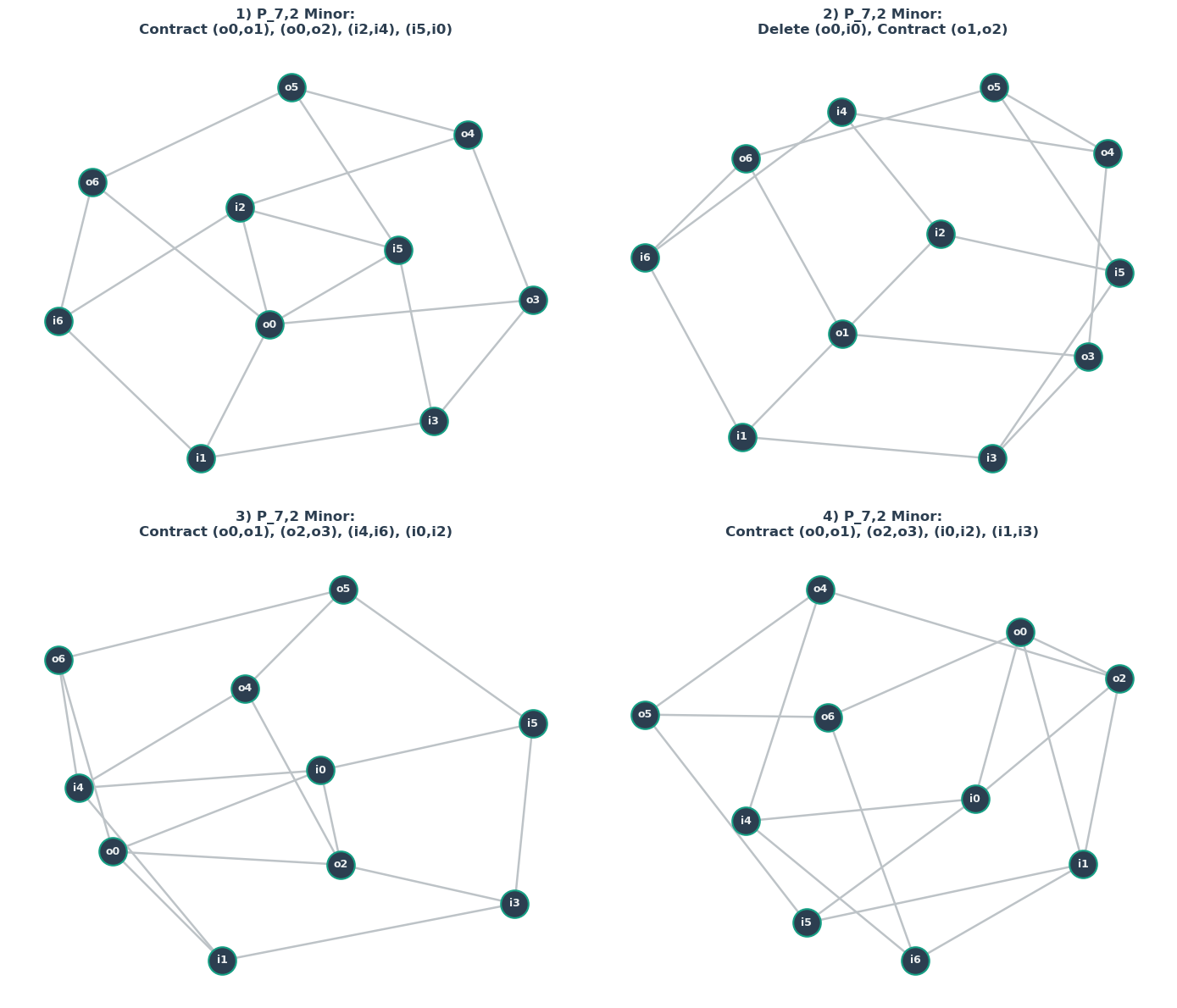}
        \caption{Minors of $P_{7,2}$ contained in $\forbiddenminors$}
        \label{fig:forbiddeniminors}
    \end{figure}
\end{example}

\section{Conclusion and future directions}\label{sec:conclusion}

In the first part, we studied the existence and complexity of forward cycle bases. We characterized when a digraph admits a forward weakly fundamental (equivalently, integral, undirected, or directed) cycle basis, reducing the question to a simple block-structure condition (Proposition~\ref{prop: weakly fundamental forward cycle basis}). For the strictly fundamental case, we showed that a strongly connected digraph admits a forward fundamental cycle basis if and only if the number of directed cycles equals the cycle rank; when it exists, this basis is unique and both its construction and the certification of non-existence run in polynomial time (Theorem~\ref{thm: forward fundamental existence}). Despite this tractability, we showed that the minimum-weight forward weakly fundamental cycle basis problem is APX-hard, via an L-reduction from the non-forward problem with metric weights (Theorem~\ref{thm:APX-hard}); Table~\ref{tab:complexity} summarizes the resulting complexity landscape. The complexity of the minimum-weight \emph{integral} cycle basis problem, forward or not, remains open and motivates the paper's second part.

There, we introduced \emph{\mcbeq{}} graphs, those for which a minimum cycle
basis is integral for every choice of edge weights, and showed this family is
minor-closed (Theorem~\ref{thm: optin is minor closed}), hence
characterized by a finite set \(\forbiddenminors\) of forbidden minors via
the Robertson--Seymour theorem. We developed a recursive algorithm
(Algorithm~\ref{alg:X search}) to test membership in
\(\forbiddenminors\).  We used this algorithm to give a complete
characterization: \(K_n\) is \mcbeq\ if and only if \(n\le7\) (Theorem~\ref{thm:complete-characterization}). A
computational search seeded at the generalized Petersen graph
\(P_{7,2}\) further produced several additional graphs in
\(\forbiddenminors\) (Figure~\ref{fig:forbiddeniminors}).

Several questions remain open. Beyond complete graphs, \(\forbiddenminors\)
is far from fully determined: the search seeded at \(P_{7,2}\) exhibits
several of its members, but no complete description is known, and it
remains open whether \(\forbiddenminors\) admits a useful structural
characterization. Resolving this is the primary open problem raised by
this work and would be a major step toward a better understanding of integral cycle bases.

\paragraph{Acknowledgements.} Niels Lindner: This research was conducted in the MobilityLab of the Research Campus MODAL funded by the German Federal Ministry of Research, Technology and Space (BMFTR), fund number 05M25KEM.

\bibliographystyle{plain}
\bibliography{mainbib}

\end{document}